\documentclass[a4paper,reqno]{amsart}

\usepackage{graphicx}
\usepackage{tikz}
\usetikzlibrary{patterns,shapes,decorations.pathmorphing,
decorations.pathreplacing,calc,arrows}
\usepackage{verbatim}
\usepackage{amssymb}
\usepackage{amstext}
\usepackage{amsmath}
\usepackage{amscd}
\usepackage{amsthm}
\usepackage{amsfonts}
\usepackage{enumerate}
\usepackage{latexsym}
\usepackage{mathrsfs}
\usepackage{hyperref}
\usepackage{mathtools}
\usepackage[all]{xy}
\usepackage{pgfplots}

\xyoption{all}

\date{\today}

\usepackage{pstricks}
\usepackage{lscape}
\usepackage{comment}

\newtheorem{theorem}{Theorem}[section]

\newtheorem{corollary}[theorem]{Corollary}
\newtheorem{lemma}[theorem]{Lemma}
\newtheorem{proposition}[theorem]{Proposition}
\newtheorem{definition-proposition}[theorem]{Definition-Proposition}
\newtheorem{definition-lemma}[theorem]{Definition-Lemma}
\newtheorem{problem}[theorem]{Problem}

\theoremstyle{definition}
\newtheorem{definition}[theorem]{Definition}

\newtheorem{example}[theorem]{Example}

\newtheorem{caution}[theorem]{Caution}

\renewcommand{\AA}{\mathcal{A}}

\newcommand{\CC}{\mathcal{C}}

\newcommand{\KKK}{\mathsf{K}}

\newcommand{\FF}{\mathcal{F}}

\newcommand{\LL}{\mathcal{L}}

\renewcommand{\SS}{\mathcal{S}}
\newcommand{\TT}{\mathcal{T}}

\newcommand{\WW}{\mathcal{W}}
\newcommand{\XX}{\mathcal{X}}

\newcommand{\Z}{\mathbb{Z}}

\newcommand{\R}{\mathbb{R}}

\newcommand{\bo}{\operatorname{b}\nolimits}

\DeclareMathOperator{\Hom}{\operatorname{Hom}}

\newcommand{\End}{\operatorname{End}\nolimits}

\newcommand{\RHom}{\mathbf{R}\strut\kern-.2em\operatorname{Hom}\nolimits}

\newcommand{\rmt}{\mathrm{t}}
\newcommand{\rmw}{\mathrm{w}}
\newcommand{\rmf}{\mathrm{f}}

\newcommand{\rmN}{\mathrm{N}}

\DeclareMathOperator{\moduleCategory}{\mathsf{mod}} \renewcommand{\mod}{\moduleCategory}

\DeclareMathOperator{\proj}{\mathsf{proj}}

\DeclareMathOperator{\inj}{\mathsf{inj}}

\DeclareMathOperator{\Sub}{\mathsf{Sub}}

\DeclareMathOperator{\simple}{\mathsf{sim}}

\DeclareMathOperator{\TF}{\mathsf{TF}}

\DeclareMathOperator{\conv}{\mathsf{conv}}

\DeclareMathOperator{\twosilt}{2\mathsf{-silt}}

\DeclareMathOperator{\twopresilt}{2\mathsf{-psilt}}

\DeclareMathOperator{\supp}{\mathsf{supp}}

\newcommand{\cut}{\ar@{-}@[|(5)]}

\DeclareMathOperator{\brick}{\mathsf{brick}}
\DeclareMathOperator{\Filt}{\mathsf{Filt}}

\DeclareMathOperator{\Fac}{\mathsf{Fac}}

\DeclareMathOperator{\Facet}{\mathsf{Facet}}
\DeclareMathOperator{\Face}{\mathsf{Face}}

\renewcommand{\epsilon}{\varepsilon}

\numberwithin{equation}{section}

\begin{document}
\title{$M$-TF equivalences on the real Grothendieck groups}

\author{Sota Asai} 
\address{Sota Asai: Graduate School of Mathematical Sciences,
University of Tokyo,  
3-8-1 Komaba, Meguro-ku, Tokyo-to, 153-8914, Japan}
\email{sotaasai@g.ecc.u-tokyo.ac.jp}

\author{Osamu Iyama}
\address{Osamu Iyama: Graduate School of Mathematical Sciences,
University of Tokyo,  
3-8-1 Komaba, Meguro-ku, Tokyo-to, 153-8914, Japan}
\email{iyama@ms.u-tokyo.ac.jp}

\begin{abstract}
For an abelian length category $\AA$ 
with only finitely many isoclasses of simple objects,
we have the wall-chamber structure and the TF equivalence 
on the dual real Grothendieck group 
$K_0(\AA)_\R^*=\Hom_\R(K_0(\AA)_\R,\R)$, 
which are defined by semistable subcategories 
and semistable torsion pairs in $\AA$ 
associated to elements $\theta \in K_0(\AA)_\R^*$.
In this paper, we introduce the $M$-TF equivalence 
for each object $M \in \AA$ 
as a systematic way to coarsen the TF equivalence.
We show that the set $\Sigma(M)$ of closures 
of $M$-TF equivalence classes
is a rational generalized fan in $K_0(\AA)_\R^*$
which is finite and complete. 
More precisely, we show that $\Sigma(M)$ is 
the normal generalized fan of the Newton polytope $\rmN(M)$
in $K_0(\AA)_\R$. When $\AA$ is the category of finitely generated modules over a finite dimensional algebra $A$, $\Sigma(M)$ can be regarded as a completion of a certain coarsening of the $g$-fan of $A$.
\end{abstract}

\maketitle 
\setcounter{tocdepth}{1}
\tableofcontents

\section{Introduction}

\subsection{Background}
The recent developments in representation theory of algebras revealed the hidden categorical/combinatorial structure of tilting theory \cite{AIR,BY,IT,KY}.
For each finite dimensional algebra $A$, there exists a non-singular fan $\Sigma_A$ (called the \emph{$g$-fan}) in the real Grothendieck group $K_0(\proj A)_\R$, whose maximal cones correspond bijectively with the isomorphism classes of basic 2-term silting complexes of $A$ \cite{AHIKM}.
Regarding elements of $K_0(\proj A)_\R$ as stability conditions \cite{K,BKT}, the first author \cite{A} introduced a certain equivalence relation on $K_0(\proj A)_\R$ called the \emph{TF equivalence}. He proved that the full dimensional TF equivalence classes are exactly the interiors of the maximal cones of $\Sigma_A$ \cite[Theorem 3.17]{A}. Thus TF equivalence classes can be regarded as a completion of the $g$-fan.
Moreover, the authors proved in \cite[Theorems 1.1 and 1.3]{AsI} (see also \cite{HYa1,HYa2,HYu}) that there is a non-trivial connection between TF equivalence classes and Kac-type canonical decompositions of elements of $K_0(\proj A)_\R$ \cite{Kac,CS,DF,P}.
Moreover, a number of recent articles studied fans appearing in representation theory, e.g.\  \cite{AFP+,BPPW,BCD+,Kai,IW,PPPP,Sh}.

It is in general hard to give an explicit description of the TF equivalence classes.
Since each TF equivalence class $E$ is clearly closed under addition and multiplication by positive scalars,
the following problem is naturally posed.

\begin{problem}\label{Prob_TF}
Let $A$ be a finite dimensional algebra, and $E$ a TF equivalence class.
\begin{enumerate}[\rm(a)] 
\item Is the closure $\overline{E}$ of $E$ a polyhedral cone? Is $\overline{E}$ simplicial?
\item Is $E$ open in $\overline{E}$?
\end{enumerate}
\end{problem}

Problem \ref{Prob_TF} is known to be true if one of the following conditions is satisfied.
\begin{enumerate}[$\bullet$]
\item (Proposition \ref{silting TF}) $E$ is contained in the $g$-fan.
\item \cite[Theorem 1.3]{AsI} $E$ contains an element in $K_0(\proj A)$, and $A$ is either $E$-tame or hereditary.
\end{enumerate}
However Problem \ref{Prob_TF} remains open in full generality.

The aim of this paper is to introduce a family of coarsening of the TF equivalence called \emph{$M$-TF equivalences} on $K_0(\proj A)_\R$, which give rise to complete generalized fans. The original TF equivalence can be regarded as a limit of $M$-TF equivalences for finitely generated $A$-modules $M \in \mod A$ (Proposition \ref{M-TF indec}(b)). In particular, if $A$ is $g$-finite, then the TF equivalence coincides with $M$-TF equivalence for some $M$ (Proposition \ref{M-TF brick all}).
We explain our results in more detail in the next subsection.

\subsection{Our results}
\emph{Throughout this paper, let $\AA$ be an abelian length category 
with only finitely many isoclasses $S(1),\ldots,S(n)$ of simple objects.}
The objects $S(1),\ldots,S(n)$ give a canonical basis
$[S(1)],\ldots,[S(n)]$ of 
the \emph{real Grothendieck group} $K_0(\AA)_\R$,
so $K_0(\AA)_\R=\bigoplus_{i=1}^n \R[S(i)]$ is 
naturally identified with the Euclidean space $\R^n$.
We also consider the dual real Grothendieck group
$K_0(\AA)_\R^*=\Hom_\R(K_0(\AA),\R)$,
on which we mainly focus in this paper rather than $K_0(\AA)_\R$.
Each element $\theta \in K_0(\AA)_\R^*$ is just an $\R$-linear form 
$\theta \colon K_0(\AA)_\R \to \R$.
Thus for any $M \in \AA$, we get a real number $\theta(M) \in \R$.
This enables us to define several subcategories of $\AA$
and subsets of $K_0(\AA)_\R^*$.

For each $\theta \in K_0(\AA)_\R^*$,
an object $M \in \AA$ is said to be \emph{$\theta$-semistable}
if $\theta(M)=0$ and 
each factor object $N$ of $M$ satisfies the linear inequality 
$\theta(N) \ge 0$.
These \emph{stability conditions} were 
originally introduced by King \cite{K} 
to characterize a condition 
in geometric invariant theory of quiver representations.
The \emph{$\theta$-semistable subcategory} $\WW_\theta \subset \AA$ 
consisting of all $\theta$-semistable objects 
is an abelian category of finite length \cite{Rudakov,HR}.
Later, Baumann-Kamnitzer-Tingley \cite{BKT} associated two \emph{semistable torsion pairs}
$(\overline{\TT}_\theta,\FF_\theta)$ 
and $(\TT_\theta,\overline{\FF}_\theta)$
to each $\theta \in K_0(\AA)_\R^*$
so that $\TT_\theta \subset \overline{\TT}_\theta$,
$\FF_\theta \subset \overline{\FF}_\theta$ and 
$\WW_\theta=\overline{\TT}_\theta \cap \overline{\FF}_\theta$,
motivated by the study of the canonical bases of quantum groups via preprojective algebras.

For each nonzero object $M \in \AA$,
Bridgeland \cite{Bridgeland} and Br\"ustle-Smith-Treffinger \cite{BST}
defined the rational polyhedral cone $\Theta_M \subset K_0(\AA)_\R^*$
consisting of all $\theta \in K_0(\AA)_\R^*$ 
such that $M$ are $\theta$-semistable.
By considering the subsets $\Theta_M$ 
for all nonzero objects $M$ as walls,
\cite{BST,Bridgeland} introduced the wall-chamber structure
(or the stablity scattering diagram)
in the Euclidean space $K_0(\AA)_\R^*$.
Then they and Yurikusa \cite{Y} showed that
this wall-chamber structure is strongly related 
to the $g$-vector fan in cluster theory \cite{FZ}
and $\tau$-tilting theory \cite{AIR,DIJ}.

To study this wall-chamber structure more,
the first named author \cite{A} defined the \emph{TF equivalence}
as an equivalence relation on $K_0(\AA)_\R^*$
so that $\theta$ and $\eta$ are TF equivalent 
if $(\overline{\TT}_\theta,\FF_\theta)=(\overline{\TT}_\eta,\FF_\eta)$ and 
$(\TT_\theta,\overline{\FF}_\theta)=(\TT_\eta,\overline{\FF}_\eta)$.
Two different points $\theta$ and $\eta$ are TF equivalent
if and only if the line segment $[\theta,\eta]$ 
does not cross any wall transversely;
see Proposition \ref{not TF stable} for details (also \cite[Theorem 2.17]{A}).
In general, the wall-chamber structure and the TF equivalence on 
$K_0(\AA)_\R^*$ often become very complicated,
but in many situations, 
it is enough to get only important information on TF equivalences
depending on purposes.

Therefore in this paper, we introduce
the \emph{$M$-TF equivalence} on $K_0(\AA)_\R^*$ 
for each object $M \in \AA$
as a systematic way to coarsen the TF equivalence.
For $M \in \AA$ and $\theta \in K_0(\AA)_\R^*$,
the semistable torsion pairs 
$(\overline{\TT}_\theta,\FF_\theta)$ and 
$(\TT_\theta,\overline{\FF}_\theta)$
give the \emph{canonical sequences}
\begin{align*}
&0 \to \overline{\rmt}_\theta M \to M  \to \rmf_\theta M \to 0
\quad (\overline{\rmt}_\theta M \in \overline{\TT}_\theta, \ 
\rmf_\theta M \in \FF_\theta), \\
&0 \to \rmt_\theta M \to M \to \overline{\rmf}_\theta M \to 0
\quad (\rmt_\theta M \in \TT_\theta, \ 
\overline{\rmf}_\theta M \in \overline{\FF}_\theta)
\end{align*}
with $\rmt_\theta M \subset 
\overline{\rmt}_\theta M \subset M$ subobjects.
The subfactor $\rmw_\theta M:=\overline{\rmt}_\theta M/\rmt_\theta M$
is in $\WW_\theta$.
On the other hand, for any object $N \in \WW_\theta$, we denote by $\supp_\theta N$ the set of isoclasses of composition factors of $N$ in $\WW_\theta$. 
Then we introduce the $M$-TF equivalence as follows.

\begin{definition}[Definition \ref{define M-TF}]\label{define M-TF intro}
Let $M \in \AA$. 
\begin{enumerate}[\rm (a)]
\item
Let $\theta,\eta \in K_0(\AA)_\R^*$.
Then we say that $\theta$ and $\eta$ are 
\emph{$M$-TF equivalent} if the following conditions are satisfied.
\begin{enumerate}[\rm (i)]
\item
The conditions
$\rmt_\theta M=\rmt_\eta M$, $\rmw_\theta M=\rmw_\eta M$ and 
$\rmf_\theta M=\rmf_\eta M$ hold, or equivalently, 
$(\overline{\rmt}_\theta M,\rmf_\theta M)=
(\overline{\rmt}_\eta M,\rmf_\eta M)$ and 
$(\rmt_\theta M,\overline{\rmf}_\theta M)=
(\rmt_\eta M,\overline{\rmf}_\eta M)$ hold.
\item
We have $\supp_\theta(\rmw_\theta M)=\supp_\eta(\rmw_\eta M)$.
Under the condition (i), this is equivalent to that
the subobjects of $\rmw_\theta M$ in $\WW_\theta$ are 
precisely the subobjects of $\rmw_\eta M$ in $\WW_\eta$.
\end{enumerate}
\item
We define $\TF(M)$ as the set of all $M$-TF equivalence classes, and 
\begin{align*}
\Sigma(M):=\{ \overline{E} \mid E \in \TF(M) \}
\end{align*}
as the set of the closures of all $E \in \TF(M)$.
\end{enumerate}
\end{definition}

By definition, $\theta$ and $\eta$ are TF equivalent
if and only if they are $M$-TF equivalent for all objects $M \in \AA$.
When $\AA$ is the category $\mod A$ of finitely generated modules over a finite dimensional algebra $A$, $\Sigma(M)$ can be regarded as a completion of a certain coarsening of the $g$-fan of $A$;
see \cite{AHIKM}.
For the notion of coarsening of fans, we refer to \cite{AHIKM}.

\begin{proposition}
(Proposition \ref{connection to g-fan})
Let $A$ be a finite dimensional algebra and $M\in\mod A$. Then $\Sigma(M)$ is a completion of the coarsening of the $g$-fan $\Sigma_A$ with respect to the equivalence relation $\sim_M$.
\end{proposition}

In this paper, we use the notion of a \emph{generalized fan} 
(see Definition \ref{define fan}) \cite{CLS}, 
which is more general than the usual notion of a fan, 
since the cones are not assumed to be strongly convex. 

The following is the main result of this paper,
where $\sigma^\circ$ denotes the \emph{relative interior} of $\sigma$, 
that is, the interior of $\sigma$
as a subset of the $\R$-vector space $\R \sigma$ spanned by $\sigma$.

\begin{theorem}[Theorems \ref{TF(M) Sigma(M)}, \ref{Sigma(M)=Sigma(N(M))}(a)
and Corollary \ref{M-TF decompose}(b)]
\label{Sigma(M)=Sigma(N(M)) intro 1}
Let $M \in \AA$.
\begin{enumerate}[\rm (a)]
\item
The set $\Sigma(M)$ is a rational generalized fan in $K_0(\AA)_\R^*$ which is finite and complete. Moreover,
we have mutually inverse bijections 
\begin{align*}
\TF(M) \simeq \Sigma(M)
\end{align*}
given by $E \mapsto \overline{E}$ for any $E \in \TF(M)$
and $\sigma \mapsto \sigma^\circ$ for any $\sigma \in \Sigma(M)$.
\item
Let $\sigma \in \Sigma(M)$.
Then the face decomposition
\begin{align*}
\sigma=\bigsqcup_{\tau \in \Face \sigma}\tau^\circ
\end{align*}
is also the decomposition to $M$-TF equivalence classes.
\end{enumerate}
\end{theorem}

We remark that the condition (ii) $\supp_\theta M=\supp_\eta M$ 
in Definition \ref{define M-TF intro}(a) 
is necessary to make $\Sigma(M)$ a generalized fan; see Example \ref{A2}.

To prove Theorem \ref{Sigma(M)=Sigma(N(M)) intro 1}, 
we need the Newton polytope $\rmN(M)$ \cite{BKT,Fei1}
in the ordinary real Grothendieck group $K_0(\AA)_\R$,
which is given as follows.

\begin{definition}[Definition \ref{define newton}]
Let $M \in \AA$.
The \emph{Newton polytope} $\rmN(M)$ is defined as
\begin{align*}
\rmN(M):=
\conv \{[X] \mid \text{$X$ is a subobject of $M$}\} \subset K_0(\AA)_\R.
\end{align*}
\end{definition} 

In this paper, 
we show that $\Sigma(M)$ is the normal generalized fan of $\rmN(M)$.
In general, the normal generalized fan of a polytope in $\R^n$
is the generalized fan defined in the following way.

\begin{definition-proposition}[Definition-Proposition \ref{Sigma(P)}]
For a polytope $P$ in $V:=\R^n$, we denote by $\Sigma(P)$ the normal generalized fan of $P$ in $V^*=\Hom_\R(V,\R)$. 
Each $\theta\in V^*$ gives a face of $P$:
\begin{align*}
P_\theta:=\{v\in P\mid\theta(v)=\max\theta(P)\}.
\end{align*}
Each face $F$ of $P$ gives a polyhedral cone in $\Sigma(P)$:
\begin{align*}
\sigma_F:=\{\theta\in V^*\mid F\subset P_\theta\}.
\end{align*}
These correspondences induce bijections between $\Sigma(P)$
and the set $\Face P$ of (nonempty) faces of $P$.
\end{definition-proposition}

Under these preparations, we can state the following result.
This can be regarded as a generalization of \cite[Theorem 5.23]{AHIKM}
(see also \cite[Theorem 5.17, Corollary 7.9]{Fei2}),
and implies Theorem \ref{Sigma(M)=Sigma(N(M)) intro 1}.

\begin{theorem}[Theorem \ref{Sigma(M)=Sigma(N(M))}]
\label{Sigma(M)=Sigma(N(M)) intro 2}
Let $M \in \AA$.
\begin{enumerate}[\rm (a)]
\item
The set $\Sigma(M)$ coincides with the normal generalized fan $\Sigma(\rmN(M))$.
\item
There exist mutually inverse order-reversing bijections
\begin{align*}
\Face \rmN(M) \leftrightarrow \Sigma(M)
\end{align*}
given by $F \mapsto \sigma_F$ for each $F \in \Face \rmN(M)$ and 
$\sigma \mapsto \rmN(M)_\theta$ for each $\sigma \in \Sigma(M)$ 
taking $\theta \in \sigma^\circ$.
Moreover each $F \in \Face \rmN(M)$ satisfies
$\dim_\R \sigma_F=n-\dim_\R F$. 
\end{enumerate}
\end{theorem}

Since $\Sigma(M)$ is finite and complete,
maximal cones in $\Sigma(M)$ are precisely
$n$-dimensional cones in $\Sigma(M)$.
Thus we write $\Sigma_n(M)$ 
for the set of all maximal cones in $\Sigma(M)$,
which is in bijection with the set $\Face_0 \rmN(M)$
of all vertices in $\Sigma(M)$ 
by Theorem \ref{Sigma(M)=Sigma(N(M)) intro 2}(b).
We apply our results to study the set $\Facet \sigma$
of \emph{facets} of each $\sigma \in \Sigma_n(M)$.
Below we introduce a decomposition
$\Facet \sigma=\Facet^+ \sigma \sqcup \Facet^- \sigma$
in terms of the bijections $\Sigma(M) \leftrightarrow \Face \rmN(M)$.

Let $\sigma \in \Sigma_n(M)$.
Then $\sigma^\circ$ is an open $M$-TF equivalence class 
by Theorem \ref{Sigma(M)=Sigma(N(M)) intro 1}(b).
Setting $\rmt_\sigma M:=\rmt_\theta M$ and 
$\rmf_\sigma M:=\rmf_\theta M$ 
by taking any $\theta \in \sigma^\circ$, we get
\begin{align*}
\sigma^\circ&=\{\theta \in K_0(\AA)_\R^* \mid
\rmt_\sigma M \in \TT_\theta, \ \rmf_\sigma M \in \FF_\theta\},\\
\partial \sigma&=
\{\theta \in \sigma \mid
\text{$\rmt_\sigma M \notin \TT_\theta$ or 
$\rmf_\sigma M \notin \FF_\theta$}\}.
\end{align*}
Thus for the purpose above, 
we define two subsets $\partial^+ \sigma$
and $\partial^- \sigma$ of the boundary $\partial \sigma$ by
\begin{align*}
\partial^+ \sigma
=\{\theta \in \sigma \mid \rmt_\sigma M \notin \TT_\theta\}
\quad \text{and} \quad
\partial^- \sigma
=\{\theta \in \sigma \mid \rmf_\sigma M \notin \FF_\theta\},
\end{align*} 
Then we have $\partial \sigma=
\partial^+ \sigma \cup \partial^- \sigma$,
and $\partial^\pm \sigma$ are
unions of faces of $\sigma$;
see Proposition \ref{union of faces}.
Thus we define 
\begin{align*}
\Facet^\pm \sigma:=\{\tau \in \Facet \sigma
\mid \tau \subset \partial^\pm \sigma\}.
\end{align*}
The following is our result 
on facets of $\sigma \in \Sigma_n(M)$.
In particular, $\partial^\pm \sigma$ have purity by (b).

\begin{theorem}[Theorem \ref{facet +-}]\label{facet +- intro}
For $M \in \AA$, let $\sigma \in \Sigma_n(M)$ and $\{v\}\in\Face_0 \rmN(M)$ the corresponding vertex.
\begin{enumerate}[\rm (a)]
\item 
$\Facet \sigma=\Facet^+ \sigma \sqcup \Facet^- \sigma$.
\item
We have $\partial^+\sigma=\bigcup_{\tau\in\Facet^+\sigma}\tau$ and $\partial^-\sigma=\bigcup_{\tau\in\Facet^-\sigma}\tau$.
\item Let $\sigma'\in\Sigma_n(M)$ such that $\tau:=\sigma\cap\sigma'\in\Facet\sigma$, and $\{v'\}\in\Face_0 \rmN(M)$ the corresponding vertex to $\sigma'$.
Then precisely one of the following statements holds.
\begin{enumerate}[\rm(i)]
\item $v>v'$ and $\tau\in\Facet^+\sigma$.
\item $v<v'$ and $\tau\in\Facet^-\sigma$.
\end{enumerate}
\end{enumerate}
\end{theorem}

The organization of this paper is as follows.
We recall the definitions and basic properties of 
polytopes, cones and fans in $\R^n$ in Subsection \ref{cones and fans},
and stability conditions, the wall-chamber structure, 
semistable torsion pairs and the TF equivalence 
in Subsection \ref{semistable pre}.
The definition and our main results on $M$-TF equivalences are stated
in Subsection \ref{Subsec def M-TF}.
Some examples are given in Subsection \ref{Subsec example}.
The proofs of our main results are given in Subsection \ref{Subsec proof}.
Finally, we study maximal cones in $\Sigma(M)$ in Section \ref{Sec max}.

\subsection*{Acknowledgment}
S.A. was supported by JSPS KAKENHI Grant Numbers JP23K12957. 
O.I. was supported by JSPS KAKENHI Grant Numbers JP23K22384.

\section{Preliminary}\label{Section Pre}

\subsection{Preliminaries on polytopes, cones and fans}
\label{cones and fans}
In this subscection,
we recall some fundamental materials on fans and polytopes. 
We refer the reader to e.g.\ \cite{Fu,BR,BP} for these materials. 

For any subset $X \subset \R^n$, 
the symbol $V_X$ denotes the $\R$-vector subspace spanned 
by $\{w-v \mid v,w \in X\}$,
and we set $A_X:=v+V_X$ taking $v \in X$. 
The set $A_X$ does not depend on the choice of $v$,
and is the smallest affine subspace containing $X$.
We write $X^\circ$ for its \emph{relative interior},
that is, the interior of $X$ as a subset of $A_X$.
For any finite set $X \subset \R^n$,
we write $\conv X$ for the convex hull of $X$.
In particular, $\conv \emptyset=\{0\}$ holds.

We here prepare some symbols and terminology on polytopes in $\R^n$.

\begin{definition}\label{polytope def}
We define the following notions.
\begin{enumerate}[\rm (a)]
\item 
A subset $P \subset \R^n$ is called a \emph{polytope} 
(resp.~a \emph{rational polytope}) in $\R^n$
if there exists a finite subset $S \subset \R^n$ 
(resp.~$S \subset \Z^n$) such that $P=\conv S$.
If $P$ is a polytope in $\R^n$, its dimension is defined by
$\dim_\R P:=\dim_\R A_P=\dim_\R V_P$.
\item
Let $P$ be a polytope in $\R^n$.
Then a subset $F \subset P$ is called a \emph{face} of $P$
if the following hold:
\begin{enumerate}[\rm (i)]
\item
$F \ne \emptyset$, and $F$ is convex.
\item
If $v,w \in P$ satisfy $[v,w] \cap F \ne \emptyset$, then $v,w \in F$.
\end{enumerate}
In particular, $P$ itself is a face of $P$, and each face is a polytope again.
We set $\Face P$ as the set of all faces of the polytope $P$,
and $\Face_i P$ as the set of all faces of dimension $i$.
We call each element of $\Face_0 P$ a \emph{vertex} of $P$,
and each element of $\Face_1 P$ an \emph{edge} of $P$.
\end{enumerate}
\end{definition}

We next define polyhedral cones and fans in $\R^n$.

\begin{definition}
We define the following notions.
\begin{enumerate}[\rm (a)]
\item 
A subset $C \subset \R^n$ is said to be 
a \emph{polyhedral cone} (resp.~a \emph{rational polyhedral cone}) if
there exist finitely many elements $v_1,\ldots,v_m \in \R^n$ 
(resp.~$v_1,\ldots,v_m \in \Z^n$)
such that $C=\sum_{i=1}^m \R_{\ge 0}v_i$.
If $C$ is a polyhedral cone in $\R^n$, its dimension is defined by
$\dim_\R P:=\dim_\R A_P=\dim_\R V_P$.
\item 
If $C$ is a polyhedral cone in $\R^n$, 
a subset $F \subset C$ is called a \emph{face} of $C$
if the following conditions are satisfied:
\begin{enumerate}[\rm (i)]
\item
$F \ne \emptyset$, $\R_{\ge 0}F \subset F$.
\item
If $v,w \in C$ satisfy $v+w \in F$, then $v,w \in F$.
\end{enumerate}
In this case, $C$ itself is a face of $C$, 
and each face $F$ is a polyhedral cone again.
We denote by $\Face C$ the set of all faces of the cone $C$, 
and by $\Face_iC$ the set of all faces of dimension $i$.
\end{enumerate}
\end{definition}

In particular, $\{0\}$ is a rational polyhedral cone in $\R^n$ as $m=0$.
Moreover, the relative interior $C^\circ$ of a cone $C$ is well-defined
as the interior of $C$ in $V_C$.

The following notion is crucial in our study.

\begin{definition}\label{define fan}
We define the following notions.
\begin{enumerate}[\rm(a)]
\item
\cite[Definition 6.2.2]{CLS} 
A \emph{generalized fan} $\Sigma$ in $\R^n$ means 
a set of polyhedral cones in $\R^n$ satisfying 
\begin{enumerate}[\rm (i)]
\item 
$\Face \sigma \subset \Sigma$ for each $\sigma \in \Sigma$, and
\item 
$\sigma \cap \tau \in \Face \sigma \cap \Face \tau$ 
for any $\sigma,\tau \in \Sigma$.
\end{enumerate}
We denote by $\Sigma_i$ 
the set of all elements of dimension $i$ in $\Sigma$.
\item 
A generalized fan $\Sigma$ in $\R^n$ is said to be 
\emph{rational} if each $\sigma \in \Sigma$ 
is a rational polyhedral cone,
\emph{finite} if $\Sigma$ is a finite set,
and \emph{complete} 
if the union $\bigcup_{\sigma \in \Sigma} \sigma$ is $\R^n$.
\end{enumerate}
\end{definition}

\begin{caution}\label{caution fan}
Notice that we do \emph{not} assume in Definition \ref{define fan} that the cones are strongly convex. 
In other words, a \emph{fan} in the usual sense is nothing but 
a generalized fan whose cones are strongly convex. 
\end{caution}

The normal generalized fan of a polytope is defined as follows, 
see \cite{BP,CLS,St}.

\begin{definition-proposition}\label{Sigma(P)}
Let $P$ be a polytope in $V:=\R^n$, and $V^*:=\Hom_\R(V,\R)$.
\begin{enumerate}[\rm (a)]
\item
For each $\theta \in V^*$, 
we define $P_\theta \in \Face P$ by 
\begin{align*}
P_\theta&:=\{ v \in P \mid \theta(v)=\max \theta(P)\}.
\end{align*}
Then 
we have
\begin{align*}
\Face P=\{ P_\theta \mid \theta \in V^*\}.
\end{align*}
\item
For any face $F \in \Face P$, we set
\begin{align*}
\sigma_F&:=\{ \theta \in V^* \mid F \subset
P_\theta\}.
\end{align*}
The \emph{normal generalized fan} of $P$ is a finite, complete, 
generalized fan in $V^*$ defined by
\begin{align*}
\Sigma(P)&:=\{ \sigma_F \mid F \in \Face P\}
=\{ \sigma_{P_\theta} \mid \theta \in V^* \}.
\end{align*}
Let $\Sigma_i(P):=\Sigma(P)_i$ for each $i \in \{0,\ldots,n\}$.
\item
We have order-reversing bijections
\begin{align*}
\Face P \simeq\Sigma(P)
\end{align*}
given by $F \mapsto \sigma_F$ for each $F \in \Face P$ and
$\sigma \mapsto P_\theta$ for each $\sigma \in \Sigma(P)$
taking $\theta \in \sigma^\circ$.
For each $i\in\{0,\ldots,n\}$, the bijections are restricted to bijections
\begin{align*}
\Face_i P \simeq\Sigma_{n-i}(P).
\end{align*}
\item
Let $\theta \in V^*$.
Then $\sigma_{P_\theta}$ is the smallest element 
of the normal generalized fan $\Sigma(P)$ which contains $\theta$.
In particular, $\sigma_{P_\theta}$ is the unique element
$\sigma \in \Sigma(P)$ such that $\theta \in \sigma^\circ$.
\end{enumerate}
\end{definition-proposition}

Clearly, $P$ has the full dimension $n$ 
if and only if $\Sigma(P)$ is a fan.
Moreover if $P$ is a lattice polytope, then
$\Sigma(P)$ is a rational generalized fan. 

Since $\Sigma(P)$ is finite and complete,
a cone $\sigma \in \Sigma(P)$ is maximal in $\Sigma(P)$
if and only if $\sigma \in \Sigma_n(M)$.
The bijection $\Face P \to \Sigma(P)$ 
in Definition-Proposition \ref{Sigma(P)}(c),
satisfies the following property.

\begin{proposition}\label{Face_d N(M) bij}
Let $P$ be a polytope in $V:=\R^n$.
Let $\sigma=\bigcap_{i=1}^m \sigma_i \in \Sigma$
with each $\sigma_i \in \Sigma_n(P)$.
For each $i \in \{1,\ldots,m\}$, define $v_i$ 
as the vertex of $P$ corresponding to $\sigma_i$. 
Then the bijection $\Sigma(P) \to \Face P$ sends $\sigma$ 
to the smallest face of $P$ containing 
$v_1,\ldots,v_m$.
\end{proposition}

\subsection{Preliminaries on stability conditions 
and semistable torsion pairs}
\label{semistable pre}
Next we prepare some notions in the representation theory.
Throughout this paper, 
any subcategory is assumed to be a full subcategory.

We first recall some general properties 
of abelian categories of finite length.
Let $\CC$ be an abelian category of finite length.
Then $\CC$ has the Jordan-H\"older property.
Thus each $M \in \CC$ has a \emph{composition series}
$0=M_0 \subsetneq M_1 \subsetneq \cdots \subsetneq M_\ell=M$,
that is, a sequence where each subfactor 
$M_i/M_{i-1}$ is a simple object of $\CC$.
The subfactors $M_i/M_{i-1}$ are called 
the \emph{composition factors} of $M$ in $\CC$.
Up to reordering and isomorphisms,
the composition factors of $M$ in $\CC$ and their multiplicities 
do not depend on the choice of a composition series of $M$ in $\CC$.
Thus the \emph{length} $\ell(M)$ of $M$ is well-defined.
Composition factors play a very important role in this paper.

In this paper, we let $\AA$ be a length abelian category
with $S(1),\ldots,S(n)$ all the distinct isoclasses of simple objects.
Then the Grothendieck group $K_0(\AA)$ is
the free abelian group $\bigoplus_{i=1}^n \Z[S(i)] \simeq \Z^n$,
so the real Grothendieck group $K_0(\AA)_\R:=K_0(\AA) \otimes_\Z \R$
is the $\R$-vector space $\bigoplus_{i=1}^n \R[S(i)] \simeq \R^n$.

The main target of this paper is the dual real Grothendieck group
$K_0(\AA)_\R^*=\Hom_\R(K_0(\AA),\R)$.
We write $[S(1)]^*,\ldots,[S(n)]^* \in K_0(\AA)_\R^*$ 
for the dual basis of $[S(1)],\ldots,[S(n)] \in K_0(\AA)_\R$,
and then we have
$K_0(\AA)_\R^*=\bigoplus_{i=1}^n \R[S(i)]^* \simeq \R^n$.

By definition, each element $\theta \in K_0(\AA)_\R^*$ is 
an $\R$-linear form $\theta \colon K_0(\AA)_\R \to \R$.
In particular, for any $M \in \AA$,
we have a real number $\theta(M) \in \R$,
and if $\theta=[S(i)]^*$, then $\theta(M)$ is the multiplicity of $S(i)$
in the composition factors of $M$.
By using this, stability conditions are defined as follows.

\begin{definition}\label{stab}\cite[Definition 1.1]{K}
Let $\theta \in K_0(\AA)_\R^*$.
\begin{enumerate}[\rm (a)]
\item
Let $M \in \AA$.
Then $M$ is said to be \emph{$\theta$-semistable} if 
$\theta(M)=0$ and $\theta(N) \ge 0$ 
for any factor object $N$ of $M$.
A $\theta$-semistable object $M$ is said to be \emph{$\theta$-stable} if
moreover $M \ne 0$ and $\theta(N)>0$ 
for any nonzero factor object $N \ne 0$ of $M$.
\item
We define the \emph{$\theta$-semistable subcategory} 
$\WW_\theta \subset \AA$ as the full subcategory consisting of
all $\theta$-semistable objects.
\end{enumerate}
\end{definition}

Then we have the following basic properties;
see \cite{Rudakov,HR}.

\begin{proposition}\label{simple stable}
Let $\theta \in K_0(\AA)_\R^*$.
\begin{enumerate}[\rm (a)]
\item
$\WW_\theta$ is a wide subcategory of $\AA$, that is,
$\WW_\theta$ is closed under taking kernels, cokernels 
and extensions in $\AA$.
\item
$\WW_\theta$ is an abelian category of finite length;
hence $\WW_\theta$ satisfies the Jordan-H\"older property.
\item
For any $M \in \WW_\theta$, 
$M$ is a simple object in $\WW_\theta$ if and only if 
$M$ is $\theta$-stable.
\end{enumerate}
\end{proposition}

We write $\simple \WW_\theta$ for 
the set of all isoclasses of simple objects of $\WW_\theta$.
By (b), for each $M \in \WW_\theta$, 
the composition factors of $M$ in $\WW_\theta$ is well-defined.
Thus we define the following symbol.

\begin{definition}\label{def supp}
For any $\theta \in K_0(\AA)_\R^*$ and $M \in \WW_\theta$, we set
\begin{align*}
\supp_\theta M:=\{ \text{isoclasses of 
composition factors of $M$ in $\WW_\theta$} \} \subset \simple \WW_\theta.
\end{align*}
\end{definition}

We remark that $\supp_\theta M$ is always a finite set,
but $\simple \WW_\theta$ can be an infinite set.
We do not consider the Grothendieck group of $\WW_\theta$ in this paper.

The category $\Filt \supp_\theta M$ is 
the smallest Serre subcategory of $\WW_\theta$
containing $M$,
where $\Filt \CC \subset \AA$ is
the full subcategory of all $M \in \AA$ admitting a sequence
$0=M_0 \subset M_1 \subset \cdots \subset M_\ell=M$
with $M_i/M_{i-1} \in \CC$.

Next we recall the wall-chamber structure in $K_0(\AA)_\R^*$
given by stability conditions.

\begin{definition}\label{Theta_M}
\cite[Definition 6.1]{Bridgeland}\cite[Definition 3.2]{BST}
We set the following notions.
\begin{enumerate}[\rm (a)]
\item
Let $M \in \AA \setminus \{0\}$.
Then we define the \emph{wall} $\Theta_M \subset K_0(\AA)_\R^*$ 
associated to $M$ by 
\begin{align*}
\Theta_M:=\{ \theta \in K_0(\AA)_\R^* \mid M \in \WW_\theta \}.
\end{align*}
\item
The \emph{wall-chamber structure} on $K_0(\AA)_\R^*$ is defined to be
the structure given by all walls $\Theta_M$ 
for $M \in \AA \setminus \{0\}$.
Thus a \emph{chamber} means each connect component of 
\begin{align*}
K_0(\AA)_\R^* \setminus 
\overline{\bigcup_{M \in \AA \setminus \{0\}}\Theta_M}.
\end{align*}
\end{enumerate}
\end{definition}

For a polyhedral cone $\sigma$ in $K_0(\AA)_\R^*$,
$\R \sigma$ denotes the vector subspace spanned 
by $\sigma$ in $K_0(\AA)_\R^*$.
For any subset $\XX \subset \AA$, 
we write $\R \XX \subset K_0(\AA)_\R$
for the subspace $\R \{[X] \mid X \in \XX\}$,
and we define a subspace $H_\XX \subset K_0(\AA)_\R^*$ by
\begin{align*}
H_\XX:=\{ \theta \in K_0(\AA)_\R^* \mid 
\text{for any $X \in \XX$, $\theta(X)=0$} \}
=\{ \theta \in K_0(\AA)_\R^* \mid \theta(\R \XX)=0 \}.
\end{align*}
Under this terminology, the face structure of $\Theta_M$ is 
related to $\supp_\theta M$ as follows.

\begin{lemma}\label{Theta_M face}\cite[Lemma 2.7]{A}
Let $M \in \AA \setminus \{0\}$ and $\sigma,\tau \in \Face \Theta_M$.
\begin{enumerate}[\rm (a)]
\item
For any $\theta \in \sigma^\circ$, the set $\supp_\theta M$ is constant,
and we have 
\begin{align*}
\sigma=\Theta_M \cap H_{\supp_\theta M}, \quad
\R\sigma=H_{\supp_\theta M}, \quad
\dim_\R \sigma=n-\dim_\R (\R \supp_\theta M).
\end{align*}
\item
Let $\theta \in \sigma^\circ$ and $\eta \in \tau^\circ$.
Then $\sigma \subset \tau$ holds if and only if
$\supp_\eta M \subset \WW_\theta$,
and $\sigma=\tau$ holds if and only if
$\supp_\eta M=\supp_\theta M$.
\end{enumerate}
\end{lemma}

We give an easy example of the wall-chamber structure.
In examples of this paper,
we often let $k$ be a field, $A$ be a finite-dimensional $k$-algebra, 
and the abelian category $\AA$ be the category $\mod A$ 
of finitely generated right $A$-modules.
In this setting, $S(1),\ldots,S(n)$ are the isoclasses of simple $A$-modules.

\begin{example}\label{A2 wall}
Let $A$ be the path algebra $k(1 \to 2)$.
There are exactly three isoclasses of indecomposable $A$-modules
$S(1),S(2)$ and $M=\begin{smallmatrix}1\\2\end{smallmatrix}$. 
Then the wall-chamber structure of $K_0(\mod A)_\R^*$ is depicted as follows:
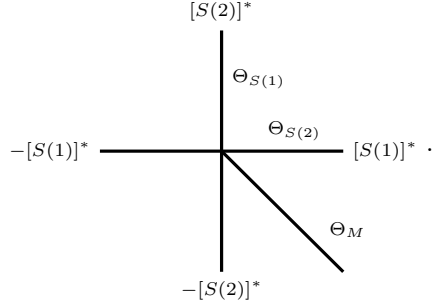
\begin{align*}
\begin{tikzpicture}[baseline=0pt,scale=0.8]
\node (00)[coordinate] at ( 0, 0) {};
\node (+0)[coordinate,label=  0:{$\scriptstyle{ [S(1)]^*}$}] at ( 2, 0) {};
\node (++)[coordinate]                                       at ( 2, 2) {};
\node (0+)[coordinate,label= 90:{$\scriptstyle{ [S(2)]^*}$}] at ( 0, 2) {};
\node (-+)[coordinate]                                       at (-2, 2) {};
\node (-0)[coordinate,label=180:{$\scriptstyle{-[S(1)]^*}$}] at (-2, 0) {};
\node (--)[coordinate]                                       at (-2,-2) {};
\node (0-)[coordinate,label=270:{$\scriptstyle{-[S(2)]^*}$}] at ( 0,-2) {};
\node (+-)[coordinate]                                       at ( 2,-2) {};
\draw[very thick] 
(0+) to[edge label={$\scriptstyle \Theta_{S(1)}$},pos=0.2] (0-);
\draw[very thick] 
(-0) to[edge label={$\scriptstyle \Theta_{S(2)}$},pos=0.8] (+0);
\draw[very thick]
(00) to[edge label={$\scriptstyle \Theta_M$},pos=0.8] (+-);
\end{tikzpicture}.
\end{align*}
The wall $\Theta_M$ is a half-line from the origin $0$,
so $\Face \Theta_M=\{\{0\},\Theta_M\}$.
For each $\theta \in \Theta_M$,
$\supp_\theta M$ is $\{S(2),S(1)\}$ if $\theta=0$,
and is $\{M\}$ otherwise.
Thus we can see that Lemma \ref{Theta_M face} holds.
\end{example}

We remark that $\Theta_M$ is not necessarily a simplicial cone
by the following example.

\begin{example}\label{not simplicial}
Define an algebra $A$, and a module $M \in \mod A$ by
\begin{align*}
A:=k \left( \begin{tikzpicture}[baseline=(0.base),scale=0.8,->]
\node (0) at ( 0, 0)   {$\vphantom{0}$};
\node (1) at ( 0, 0.5) {$1$};
\node (2) at ( 0,-0.5) {$2$};
\node (3) at ( 1, 0.5) {$3$};
\node (4) at ( 1,-0.5) {$4$};
\draw [thick] (1) to (2);
\draw [thick] (2) to (4);
\draw [thick] (1) to (3);
\draw [thick] (3) to (4);
\end{tikzpicture} \right), \quad
M:= \left( \begin{tikzpicture}[baseline=(0.base),scale=0.8,->]
\node (0) at ( 0, 0)   {$\vphantom{0}$};
\node (1) at ( 0, 0.5) {$k$};
\node (2) at ( 0,-0.5) {$k$};
\node (3) at ( 1, 0.5) {$k$};
\node (4) at ( 1,-0.5) {$k$};
\draw [thick] (1) to (2);
\draw [thick] (2) to (4);
\draw [thick] (1) to (3);
\draw [thick] (3) to[edge label={\scriptsize $\lambda$}] (4);
\end{tikzpicture} \right) \quad (\lambda \in k^\times).
\end{align*}
Then $A$ is a tame hereditary algebra,
and $M$ is a module in the mouth of some homogeneous tube 
in the Auslander-Reiten quiver.
The wall $\Theta_M$ is not simplicial, because direct calculation gives that
$\Theta_M$ has four one-dimensional faces
\begin{align*}
\R_{\ge 0}([S(1)]^*-[S(2)]^*),\R_{\ge 0}([S(1)]^*-[S(3)]^*),
\R_{\ge 0}([S(2)]^*-[S(4)]^*),\R_{\ge 0}([S(3)]^*-[S(4)]^*).
\end{align*}
\end{example}

We next recall torsion pairs.
For a class $\SS$ of objects in $\AA$, let
\begin{align*}
\SS^\perp:=\{M \in \AA \mid \Hom_\AA(\SS,M)=0\}\ \text{ and }\ 
{^\perp \SS}:=\{M \in \AA \mid \Hom_\AA(M,\SS)=0\}.
\end{align*}

\begin{definition}
Let $\TT,\FF$ be full subcategories of $\AA$.
\begin{enumerate}[\rm (a)]
\item
The subcategory $\TT$ (resp.~$\FF$) is called 
a \emph{torsion class} (resp.~a \emph{torsion-free class}) in $\AA$ 
if it is closed under factor objects (resp.~subobjects) and extensions.
\item
We say that a pair $(\TT,\FF)$ of 
a torsion class $\TT$ and a torsion-free class $\FF$ 
is a \emph{torsion pair} in $\AA$ 
if $\FF=\TT^\perp$ and $\TT={^\perp \FF}$.
\end{enumerate}
\end{definition}

Since $\AA$ is of finite length,
we can use the following well-known characterizations.

\begin{lemma}\label{torsion chara}
Let $\TT,\FF$ be full subcategories of $\AA$.
\begin{enumerate}[\rm (a)]
\item
The pair $(\TT,\FF)$ is a torsion pair in $\AA$ if and only if $\TT$ is a torsion class and $\FF=\TT^\perp$ if and only if $\FF$ is a torsionfree class and $\TT={^\perp \FF}$.
\item
Let $(\TT,\FF)$ be a torsion pair of $M \in \AA$.
Set $\rmt M$ as the largest subobject $L \subset M$ 
such that $L \in \TT$, and $\rmf M:=M/\rmt M$.
Then $\rmf M \in \FF$, so we have an exact sequence 
(called the \emph{canonical sequence})
\begin{align*}
0 \to \rmt M \to M \to \rmf M \to 0 \quad (\rmt M \in \TT, \ \rmf M \in \FF).
\end{align*}
\end{enumerate}
\end{lemma}

In this paper, we are interested only in the torsion pairs
given by the following subcategories.

\begin{definition}\label{num tors}\cite[Subsection 3.1]{BKT}
Let $\theta \in K_0(\AA)_\R^*$.
We define the following subcategories of $\AA$:
\begin{align*}
\overline{\TT}_\theta&:=\{ M \in \AA \mid 
\text{$\theta(N) \ge 0$ for any factor object $N$ of $M$} \},\\
\FF_\theta&:=\{ M \in \AA \mid 
\text{$\theta(L) < 0$ for any nonzero subobject $L \ne 0$ of $M$} \},\\
\TT_\theta&:=\{ M \in \AA \mid 
\text{$\theta(N) > 0$ for any nonzero factor object $N \ne 0$ of $M$} \},\\
\overline{\FF}_\theta&:=\{ M \in \AA \mid 
\text{$\theta(L) \le 0$ for any subobject $L$ of $M$} \}.
\end{align*}
\end{definition}

By definition, 
we get $\WW_\theta=\overline{\TT}_\theta \cap \overline{\FF}_\theta$.
It is easy to see that 
$\TT_\theta \subset \overline{\TT}_\theta$ 
are torsion classes in $\AA$, and
$\FF_\theta \subset \overline{\FF}_\theta$ 
are torsion-free classes in $\AA$.
They form torsion pairs as follows.

\begin{proposition}\label{numerical torsion pair}\cite[Proposition 3.1]{BKT}
Let $\theta \in K_0(\AA)_\R^*$.
Then $(\overline{\TT}_\theta,\FF_\theta)$ and 
$(\TT_\theta,\overline{\FF}_\theta)$ are torsion pairs in $\AA$.
\end{proposition}

Thus we call 
$(\overline{\TT}_\theta,\FF_\theta)$ and 
$(\TT_\theta,\overline{\FF}_\theta)$ 
the \emph{semistable torsion pairs} associated to 
each $\theta \in K_0(\AA)_\R^*$.
The equality $\WW_\theta=\overline{\TT}_\theta \cap \overline{\FF}_\theta$
implies that the interval $[\TT_\theta,\overline{\TT}_\theta]$
in the lattice of torsion classes in $\AA$
is a wide interval in the sense of \cite{AP}.
By using semistable torsion pairs,
the following equivalence relation is defined.

\begin{definition}\cite[Definition 2.13]{A}
Let $\theta,\eta \in K_0(\AA)_\R^*$.
We say that $\theta$ and $\eta$ are \emph{TF equivalent}
if $(\overline{\TT}_\theta,\FF_\theta)=(\overline{\TT}_\eta,\FF_\eta)$ and 
$(\TT_\theta,\overline{\FF}_\theta)=(\TT_\eta,\overline{\FF}_\eta)$.
\end{definition}

The relationship with the wall-chamber structure is given as follows.
This is a slightly stronger version of \cite[Theorem 2.17]{A}.

Here, $[\theta,\eta]$ is the line segment in $\R^n$ connecting 
$\theta$ and $\eta$,
and a \emph{brick} in $\AA$ means a nonzero object $S \in \AA$
such that any nonzero endomorphism on $S$ is isomorphic.
In particular, if $S$ is $\theta$-stable 
for some $\theta \in K_0(\AA)_\R^*$, 
then $S$ is a simple object of $\WW_\theta$, so $S$ is a brick.
In a similar way to Lemma \ref{Theta_M face},
we set $H_X:=\{ \theta \in K_0(\AA)_\R^* \mid \theta(X)=0\}$
for each $X \in \AA$.

\begin{proposition}\label{not TF stable}
Let $\theta \ne \eta \in K_0(\AA)_\R^*$.
Then the following conditions are equivalent.
\begin{enumerate}[\rm (a)]
\item
The elements $\theta$ and $\eta$ are TF equivalent.
\item
Any $\zeta \in [\theta,\eta]$ is TF equivalent to $\theta$ and $\eta$.
\item
For any $\zeta \in [\theta,\eta]$,
the subcategory $\WW_\zeta \subset \AA$ is constant.
\item
For any nonzero object $M \in \AA \setminus \{0\}$,
the intersection $\Theta_M \cap [\theta,\eta]$ is 
$[\theta,\eta]$ or $\emptyset$.
\item
There exists no brick $S \in \AA$ such that
$H_S \cap [\theta,\eta]$ is one point.
\item
There exists no brick $S \in \AA$ such that
$H_S \cap [\theta,\eta]$
has a unique element $\zeta$ and that $S$ is $\zeta$-stable.
\end{enumerate}
\end{proposition}

\begin{proof}
The conditions (a)--(e) are equivalent
by the original statement \cite[Theorem 2.17]{A}.
Moreover (e)$\Rightarrow$(f) is obvious.
It remains to show (f)$\Rightarrow$(a).

We assume that $\theta$ and $\eta$ are not TF equivalent,
and show the existence of a brick $S$ satisfying the condition in (f).

Since $\theta$ and $\eta$ are not TF equivalent,
$\overline{\TT}_\theta \cap \FF_\eta \ne \{0\}$ or 
$\overline{\FF}_\theta \cap \TT_\eta \ne \{0\}$ holds
by Lemma \cite[Lemma 2.19]{A}.
We may assume the former.
Then we can take nonzero 
$S \in \overline{\TT}_\theta \cap \FF_\eta$
such that the length $\ell(S)$ of $S$ in $\AA$ is minimum.
In this case, $\theta(S) \ge 0$ and $\eta(S)<0$ holds,
so $H_S \cap [\theta,\eta]$ is one point;
write $H_S \cap [\theta,\eta]=\{\zeta\}$.

It remains to show that $S$ is $\zeta$-stable.
Clearly $\zeta(S)=0$.
Let $0\neq X \subsetneq S$ be a proper subobject.
Then $\overline{\rmt}_\theta X \in \overline{\TT}_\theta$
is also in $\FF_\eta$, 
because $\overline{\rmt}_\theta X \subset X \subset S \in \FF_\eta$.
Thus we have 
$\overline{\rmt}_\theta X \in \overline{\TT}_\theta \cap \FF_\eta$.
By minimality of $\ell(S)$, we get $\overline{\rmt}_\theta X=0$, or equivalently,
$X \in \FF_\theta$.
Since $X \ne 0$, we have $\theta(X)<0$.
On the other hand, $X \subset S \in \FF_\eta$ imply $\eta(X)<0$.
Thus $\zeta(X)<0$ by $\zeta \in [\theta,\eta]$.
Therefore $S$ is $\zeta$-stable; hence $S$ is a brick.
\end{proof}

In the rest of this subsection, 
we record some important observations on 
the closures of TF equivalence classes,
which satisfy the following property.

\begin{lemma}\label{TF closure}\cite[Lemma 2.16]{A}
Let $E=[\theta]$ be the TF equivalence class of $\theta$ 
in $K_0(\AA)_\R^*$.
Then its closure $\overline{E}$ satisfies
\begin{align*}
\overline{E}=\{\eta \in K_0(\AA)_\R^* \mid
\overline{\TT}_\theta \subset \overline{\TT}_\eta, \ 
\overline{\FF}_\theta \subset \overline{\FF}_\eta\}.
\end{align*}
In particular, $\overline{E}$ is a union of TF equivalence classes.
\end{lemma}

As in Problem \ref{Prob_TF}, for each TF equivalence class $E$,
it is expected that $\overline{E}$ is a polyhedral cone in $K_0(\AA)_\R^*$,
and that $E$ is an open subset of $\overline{E}$.
If this holds, the following properties are useful to understand
the TF equivalence classes contained in $\overline{E}$.

\begin{theorem}\label{TF cone face}
The following statements hold.
\begin{enumerate}[\rm(a)]
\item
Let $E$ be a TF equivalence class in $K_0(\AA)_\R^*$, 
and $L$ a line segment contained in $\overline{E}$.
Then all elements in $L^\circ$ are TF equivalent.
\item
Let $C$ be a polyhedral cone in $K_0(\AA)_\R^*$ and $F\in\Face C$. 
If all elements in $C^\circ$ are TF equivalent, 
then all elements in $F^\circ$ are TF equivalent.
\item 
Let $E$ be a TF equivalence class in $K_0(\AA)_\R^*$
such that $\overline{E}$ is a polyhedral cone.
Then $\overline{E}$ is a union of finitely many TF equivalence classes.
\end{enumerate}
\end{theorem}

\begin{proof}
(a) 
Otherwise, by Proposition \ref{not TF stable}(a)$\Leftrightarrow$(f), there exists $S \in \AA$ such that
$H_S\cap L^\circ=\{\zeta\}$ and that $S$ is $\zeta$-stable and hence $\zeta\in\Theta_S^\circ$.
If $E\subset\Theta_S$, then $L\subset\overline{E}\subset\Theta_S$ hold, which contradict to $L\cap\Theta_S=\{\zeta\}$. 
Since $E$ is a TF equivalence class, $E\cap\Theta_S=\emptyset$ holds. 
Since $L^\circ \cap H_S$ is a point, we can take $\theta\in L^\circ\cap H_S^+$ and $\eta\in L^\circ\cap H_S^-$.
Take $\theta'\in E\cap H_S^+$ (resp.~$\eta'\in E\cap H_S^-$) which is sufficiently close to $\theta$ (resp.~$\eta$), and let $[\theta',\eta']\cap H_S=\{\zeta'\}$. Then $\zeta'\in\Theta_S$ holds since $\zeta\in\Theta_S^\circ$ and $\zeta'$ is sufficiently close to $\zeta$. Moreover, $\zeta'\in E$ holds by the convexity of $E$. Thus $E\cap\Theta_S\neq\emptyset$ holds, a contradiction.

(b) 
For each $\theta,\eta\in F^\circ$, 
there exists a line segment $L\subset F$ such that 
$\theta,\eta\in L^\circ$.
Since $L \subset F \subset C=\overline{C^\circ}$ and
$C^\circ$ is contained in a single TF equivalence class,
we can apply (a). 
Thus $\theta$ and $\eta$ are TF equivalent.

(c)
Since $\overline{E}$ is a polyhedral cone, 
$\overline{E}$ has only finitely many faces.
Thus (b) and Lemma \ref{TF closure} imply the assertion.
\end{proof}

\subsection{Relationship with silting theory}\label{prelim silt}

In this subsection, we recall the relationship 
between silting theory and the TF equivalence.

Let $A$ be a finite dimensional algebra over a field $k$. Let $\AA:=\mod A$, and $S(1),\ldots,S(n)$ be the isoclasses of simple $A$-modules.
Let $\proj A$ be the category of 
finitely generated right projective $A$-modules,
and $P(1),\ldots,P(n)$ the isoclasses of
indecomposable projective $A$-modules such that 
$P(i)$ is the projective cover of $S(i)$.
We regard $\proj A$ as an exact category with the split short exact sequences.
Then the Grothendieck group $K_0(\proj A)$ is
the free abelian group $\bigoplus_{i=1}^n \Z[P(i)]$,
so the real Grothendieck group $K_0(\proj A)_\R^*=K_0(\proj A) \otimes_\Z \R$
is the $\R$-vector space $\bigoplus_{i=1}^n \R[P(i)]$.
We define the \emph{Euler bilinear form} 
as the non-degenerate $\R$-bilinear form
\begin{align*}
\langle !,? \rangle \colon K_0(\proj A)_\R \times K_0(\mod A)_\R &\to \R, \\
([P],[X]) &\mapsto\dim_k\Hom_A(P,X).
\end{align*}
This is non-degenerate since $\langle[P(i)],[S(j)]\rangle=\delta_{i,j}d_j$ holds for each $i,j$, where $d_i:=\dim_k \End_A(S(i))$.
Thus we identify $K_0(\proj A)_\R$ and $K_0(\AA)_\R^*$ via the isomorphism $K_0(\proj A)_\R \to K_0(\AA)_\R^*$
given by $\theta \mapsto \langle \theta,? \rangle$.

We recall the definitions of some basic notions in silting theory.
We write $\KKK^{\bo}(\proj A)$ for the homotopy category 
of bounded complexes over $\proj A$.
A full subcategory $\CC$ of $\KKK^{\bo}(\proj A)$
is said to be \emph{thick} 
if it is closed under mapping cones, $[\pm1]$ and direct summands.

\begin{definition}\cite[5.1]{KV}
Let $U$ be a complex in $\KKK^{\bo}(\proj A)$.
We say that $U$ is 
\emph{presilting} if $\Hom_{\KKK^{\bo}(\proj A)}(U,U[>0])=0$, and \emph{silting} if $U$ is presilting and 
$\KKK^{\bo}(\proj A)$ is the smallest thick subcategory containing $U$.
We say that $U$ is \emph{2-term}
if $U$ is isomorphic to a complex $U^{-1} \to U^0$
whose terms except the $-1$st and $0$th ones are zero.
\end{definition}

We write $\twopresilt A$ (resp.~$\twosilt A$)
for the set of isoclasses of basic 2-term presilting 
(resp.~2-term silting) complexes in $\KKK^{\bo}(\proj A)$.
For each $U \in \twopresilt A$,
we denote by $|U|$ the number of isoclasses of its direct summands
in $\KKK^{\bo}(\proj A)$.
Thus $|A|=n$ holds.
For each $U \in \twopresilt A$, we define 
\begin{align*}
C^\circ(U):=\sum_{i=1}^m \R_{>0}[U_i], \quad
C(U):=\sum_{i=1}^m \R_{\ge 0}[U_i].
\end{align*}
We define the \emph{$g$-fan} on $K_0(\proj A)_\R$ by
\begin{align*}
\Sigma_A:=\{C(U) \mid U \in \twopresilt A\}.
\end{align*}
Then basic properties in silting theory can be summarized as follows \cite{AIR,Ai,DIJ,A,DF}.
We denote by $\brick A$ the set of isoclasses of bricks in $\mod A$.

\begin{proposition}\label{nonsingular}
\cite[Proposition 4.2]{AHIKM}
The $g$-fan $\Sigma_A$ satisfy the following statements.
\begin{enumerate}[\rm(a)]
\item $\Sigma_A$ is a nonsingular fan in $K_0(\proj A)_\R$.
\item Any cone in $\Sigma_A$ is 
a face of a cone of dimension $n$.
\item Any cone in $\Sigma_A$ of dimension $n-1$ 
is a face of precisely two cones of dimension $n$.
\item 
$\Sigma_A$ is complete if and only if $\Sigma_A$ is finite
if and only if $\brick A$ is finite.
\end{enumerate}
\end{proposition}

Let $\nu:=-\otimes_A(DA):\KKK^{\bo}(\proj A)\simeq\KKK^{\bo}(\inj A)$ be the Nakayama functor.
The $g$-fan is strongly related to the TF equivalence as follows; see also \cite{BST}.

\begin{proposition}\label{silting TF}
\cite[Proposition 3.3]{Y} \cite[Proposition 3.11]{A}
Let $U \in \twopresilt A$ and $\theta\in C^\circ(U)$.
\begin{enumerate}[\rm (a)]
\item
We have $(\overline{\TT}_\theta,\FF_\theta)=
({^\perp H^{-1}(\nu U)},\Sub H^{-1}(\nu U))$ and 
$(\TT_\theta,\overline{\FF}_\theta)=
(\Fac H^0(U),H^0(U)^\perp)$.
\item $C^\circ(U)$ is a TF equivalence class in $K_0(\proj A)_\R$. 
\item 
If $U \in \twosilt A$, then 
we have $(\overline{\TT}_\theta,\FF_\theta)=
(\TT_\theta,\overline{\FF}_\theta)$ and $\WW_\theta=\{0\}$.
\end{enumerate}
\end{proposition}

Therefore we have an injection
\begin{align*}
\twopresilt A \to 
\{\text{TF equivalence classes in $K_0(\proj A)_\R$}\},\ U\mapsto C^\circ(U).
\end{align*}
In particular, the set 
\begin{align*}
\overline{\Sigma}_A:=\{\overline{E}\mid \mbox{$E$: TF equivalence class}\}
\end{align*}
can be regarded as a completion of $\Sigma_A$ in the following sense:
\begin{enumerate}[\rm (a)]
\item 
The inclusion $\Sigma_A\subset\overline{\Sigma}_A$ holds.
\item 
The set $\overline{\Sigma}_A$ consists of cones in $K_0(\proj A)_\R$ such that $\bigcup_{\sigma\in\overline{\Sigma}_A}=K_0(\proj A)_\R$;
here a \emph{cone} means a nonempty subset closed under taking
addition and multiple by positive scalars.
\end{enumerate}
As in Problem \ref{Prob_TF}, it is still open whether
every cone in $\overline{\Sigma}_A$ is a polyhedral cone.

\section{$M$-TF equivalences}

Recall that $\AA$ is an abelian length category  
with only finitely many isoclasses
$S(1),\ldots,S(n)$ of simple objects.
Then the real Grothendieck group $K_0(\AA)_\R=K_0(\AA) \otimes_\Z \R$
is the $\R$-vector space $\bigoplus_{i=1}^n \R[S(i)] \simeq \R^n$.
We mainly consider its dual real Grothendieck group
$K_0(\AA)_\R^*=\Hom_\R(K_0(\AA),\R)
=\bigoplus_{i=1}^n \R[S(i)]^* \simeq \R^n$,
where $[S(1)]^*,\ldots,[S(n)]^* \in K_0(\AA)_\R^*$ 
is the dual basis of $[S(1)],\ldots,[S(n)] \in K_0(\AA)_\R$.

\subsection{The definition and the main results}\label{Subsec def M-TF}

In this subsection, for each object $M \in \AA$,
we define the $M$-TF equivalence by coarsening the TF equivalence,
and construct a generalized fan $\Sigma(M)$ 
in $K_0(\AA)_\R^*$.

For $\theta \in K_0(\AA)_\R^*$, Proposition \ref{numerical torsion pair} shows that
there exist short exact sequences
\begin{align*}
&0 \to \overline{\rmt}_\theta M \to M  \to \rmf_\theta M \to 0
\quad (\overline{\rmt}_\theta M \in \overline{\TT}_\theta, \ 
\rmf_\theta M \in \FF_\theta), \\
&0 \to \rmt_\theta M \to M \to \overline{\rmf}_\theta M \to 0
\quad (\rmt_\theta M \in \TT_\theta, \ 
\overline{\rmf}_\theta M \in \overline{\FF}_\theta).
\end{align*}
We set $\rmw_\theta M:=\overline{\rmt}_\theta M/\rmt_\theta M$.
Then we have $\rmw_\theta M \in \overline{\TT}_\theta \cap \overline{\FF}_\theta =\WW_\theta$.

Recall from Definition \ref{def supp} that we have set
\begin{align*}
\supp_\theta X:=\{ \text{isoclasses of composition factors of $X$ in $\WW_\theta$}\}
\end{align*}
for $X\in\WW_\theta$.
By using these, 
we introduce the $M$-TF equivalence for each $M \in \AA$ as follows.

\begin{definition}\label{define M-TF}
Let $M \in \AA$. 
\begin{enumerate}[\rm (a)]
\item
Let $\theta,\eta \in K_0(\AA)_\R^*$.
Then we say that $\theta$ and $\eta$ are 
\emph{$M$-TF equivalent} if the following conditions are satisfied.
\begin{enumerate}[\rm (i)]
\item
The conditions
$\rmt_\theta M=\rmt_\eta M$, $\rmw_\theta M=\rmw_\eta M$ and 
$\rmf_\theta M=\rmf_\eta M$ hold, or equivalently, 
$(\overline{\rmt}_\theta M,\rmf_\theta M)=
(\overline{\rmt}_\eta M,\rmf_\eta M)$ and 
$(\rmt_\theta M,\overline{\rmf}_\theta M)=
(\rmt_\eta M,\overline{\rmf}_\eta M)$ hold.
\item
We have $\supp_\theta(\rmw_\theta M)=\supp_\eta(\rmw_\eta M)$.
Under the condition (i), this is equivalent to that
the subobjects of $\rmw_\theta M$ in $\WW_\theta$ are precisely the subobjects of $\rmw_\eta M$ in $\WW_\eta$.
\end{enumerate}
\item We write $[\theta]_M$ for the $M$-TF equivalence class of $\theta$.
For each $E=[\theta]_M$, we set 
\begin{align*}
&\rmt_E M:=\rmt_\theta M, \quad \rmw_E M:=\rmw_\theta M, \quad 
\rmf_E M:=\rmf_\theta M, \quad
\overline{\rmt}_E M:=\overline{\rmt}_\theta M, \quad 
\overline{\rmf}_E M:=\overline{\rmf}_\theta M, \\
&\supp_E(\rmw_E M):=\supp_\theta(\rmw_\theta M).
\end{align*}
Notice that $\rmt_EX$, $\rmw_EX$ and $\rmf_EX$ for other objects $X\in\AA$ are not defined.
\item
We define $\TF(M)$ as the set of all $M$-TF equivalence classes, and 
\begin{align*}
\Sigma(M):=\{ \overline{E} \mid E \in \TF(M) \}
\end{align*}
as the set of the closures of all $E \in \TF(M)$.
\end{enumerate}
\end{definition}

We almost immediately have the following properties
from the definition of $M$-TF equivalences.
We write $\brick \AA$ for the set of isoclasses in $\AA$.

\begin{proposition}\label{M-TF indec}
Let $\theta,\eta \in K_0(\AA)_\R^*$.
\begin{enumerate}[\rm (a)]
\item
Assume $M=M_1 \oplus M_2 \in \AA$. 
Then $\theta$ and $\eta$ are $M$-TF equivalent if and only if
they are $M_1$-TF equivalent and are $M_2$-TF equivalent.
\item
The elements $\theta$ and $\eta$ are TF equivalent if and only if 
$\theta$ and $\eta$ are $M$-TF equivalent for any $M \in \AA$
if and only if 
$\theta$ and $\eta$ are $M$-TF equivalent for any $M \in \brick \AA$.
\end{enumerate}
\end{proposition}

\begin{proof}
(a) is clear, because $\rmt_\theta M=\rmt_\theta M_1 \oplus \rmt_\theta M_2$
and similar properties hold.

(b)
The two ``only if'' parts are clear.
It remains to show that, if
$\theta$ and $\eta$ are $M$-TF equivalent for any brick $M$,
then $\theta$ and $\eta$ are TF equivalent.

Let $M \in \TT_\theta$ be a brick.
Then by assumption, $M=\rmt_\theta M=\rmt_\eta M\in\TT_\eta$. Thus all bricks in $\TT_\theta$ belong to $\TT_\eta$. Since $\TT_\theta$ is the smallest torsion class containing $\TT_\theta\cap\brick \AA$ \cite[Lemma 3.9]{DIRRT}, we have $\TT_\theta \subset \TT_\eta$.
By symmetry, we get $\TT_\theta=\TT_\eta$. Dually, $\FF_\theta=\FF_\eta$ holds, and hence $\theta$ and $\eta$ are TF equivalent.
\end{proof}

In the rest of this subsection, we give our results on 
$\TF(M)$ and $\Sigma(M)$, which will be proved
in Subsection \ref{Subsec proof}.
The following first one is natural but nontrivial.

\begin{theorem}\label{TF(M) Sigma(M)}
Let $M \in \AA$. Then the following statements hold.
\begin{enumerate}[\rm (a)]
\item 
$\Sigma(M)$ is a finite set of rational polyhedral cones 
in $K_0(\AA)_\R^*$.
\item 
We have mutually inverse bijections $\TF(M) \simeq \Sigma(M)$
of finite sets given by $E \mapsto \overline{E}$ for any $E \in \TF(M)$
and $\sigma \mapsto \sigma^\circ$ for any $\sigma \in \Sigma(M)$.
\end{enumerate}
\end{theorem}

Our aim of this paper is to understand $\Sigma(M)$ 
via the Newton polytope $\rmN(M)$.

\begin{definition}\cite[Subsection 1.3]{BKT}
\cite[Definition 2.3]{Fei1}\label{define newton}
Let $M \in \AA$.
The \emph{Newton polytope} $\rmN(M)$ is defined as
\begin{align*}
\rmN(M):=
\conv \{[X] \mid \text{$X$ is a subobject of $M$}\} \subset K_0(\AA)_\R.
\end{align*}
\end{definition} 

By Definition-Proposition \ref{Sigma(P)}, we can consider 
the normal generalized fan $\Sigma(\rmN(M))$ in $K_0(\AA)_\R^*$, 
which was studied in \cite{AHIKM,BKT,Fei2}.
For each $i \in \{0,\ldots,n\}$,
we write $\Face_i \rmN(M)$ for the set of all $i$-dimensional faces
of $\rmN(M)$ as in Definition \ref{polytope def}.
In particular, $\Face_0 \rmN(M)$ (resp.~$\Face_1 \rmN(M)$)
is the set of vertices (resp.~edges) in $\rmN(M)$.

\begin{theorem}\label{Sigma(M)=Sigma(N(M))}
Let $M \in \AA$.
\begin{enumerate}[\rm (a)]
\item
We have $\Sigma(M)=\Sigma(\rmN(M))$.
In particular, $\Sigma(M)$ is a rational generalized fan
which is finite and complete.
\item
For each $i \in \{0,\ldots,n\}$, there exist mutually inverse order-reversing bijections
\begin{align*}
\Face_{n-i} \rmN(M) \leftrightarrow \Sigma_i(M)
\end{align*}
given by $F \mapsto \sigma_F$ for each $F \in \Face \rmN(M)$ and
$\sigma \mapsto \rmN(M)_\theta$ for each $\sigma \in \Sigma(M)$ 
taking $\theta \in \sigma^\circ$.
In particular, each $F \in \Face \rmN(M)$ satisfies
$\dim_\R \sigma_F=n-\dim_\R F$. 
\end{enumerate}
\end{theorem}

Since $\Sigma(M)=\Sigma(\rmN(M))$,
$\Sigma(M)$ is a fan (that is, every cone is strongly convex)
if and only if $\rmN(M)$ is full-dimensional
if and only if $M$ is sincere.
The smallest element of $\Sigma(M)$ is explicitly given 
in Proposition \ref{smallest}.

As a corollary of Theorem \ref{Sigma(M)=Sigma(N(M))}, 
we have the following properties.

\begin{corollary}\label{M-TF decompose}
Let $M \in \AA$, $E \in \TF(M)$, 
$\sigma:=\overline{E}$.
\begin{enumerate}[\rm (a)]
\item
As generalized fans, 
$\Face \sigma$ is a subfan of $\Sigma(M)$.
\item
For each $\sigma' \in \Face \sigma$,
we have $\sigma' \in \Sigma(M)$ and $(\sigma')^\circ \in \TF(M)$.
Therefore the decomposition
\begin{align*}
\sigma=\bigsqcup_{\tau \in \Face \sigma}\tau^\circ
\end{align*}
is also the decomposition to $M$-TF equivalence classes.
\item
Let $\sigma' \in \Face \sigma$, and set $E':=(\sigma')^\circ$.
Then we have 
\begin{align*}
\sigma'=\{ \theta \in \sigma \mid \theta(\supp_{E'}(\rmw_{E'} M))=0\}.
\end{align*}
\end{enumerate}
\end{corollary}

In Definition \ref{Theta_M}, 
we defined the wall-chamber structure on $K_0(\AA)_\R^*$.
Applying our results to the wall $\Theta_M$ 
for each $M \in \AA \setminus \{0\}$,
we have the following properties.

\begin{corollary}\label{Theta_M Sigma(M)}
Let $M \in \AA \setminus \{0\}$.
\begin{enumerate}[\rm(a)]
\item 
The vertices $\{0\}$ and $\{[M]\}$ in $\Face_0 \rmN(M)$ correspond to
$\sigma_-:=\{\theta \in K_0(\AA)_\R^* \mid M \in \overline{\FF}_\theta\}$
and 
$\sigma_+:=\{\theta \in K_0(\AA)_\R^* \mid M \in \overline{\TT}_\theta\}$
in $\Sigma_n(M)$, respectively.
\item 
Set $F \in \Face \rmN(M)$ as the smallest face such that $0,[M] \in F$.
Then we have $\Theta_M=\sigma_- \cap \sigma_+=\sigma_F \in \Sigma(M)$ and 
$\Theta_M^\circ \in \TF(M)$.
\item
The bijections $\Face \rmN(M) \leftrightarrow \Sigma(M)$ 
in Theorem \ref{Sigma(M)=Sigma(N(M))}(b) restrict to
\begin{align*}
\{F' \in \Face \rmN(M) \mid 0,[M] \in F'\} \leftrightarrow \Face \Theta_M.
\end{align*}
\item 
As generalized fans,
$\Face\Theta_M$ is a subfan of $\Sigma(M)$.
\item 
For each $\sigma \in \Face \Theta_M$, we have
$\sigma^\circ \in \TF(M)$ and $\sigma \in \Sigma(M)$.
\end{enumerate}
\end{corollary}

As in Example \ref{not simplicial}, $\Theta_M$ may not be simplicial,
so $\Sigma(M)$ is not necessarily a simplicial fan.

We will prove the results above in Subsection \ref{Subsec proof}.

In the rest of this subsection,
let $\AA=\mod A$ with $A$ a finite dimensional algebra over a field $k$.
As in Subsection \ref{prelim silt},
we identify $K_0(\proj A)_\R$ with $K_0(\mod A)_\R^*$
via the Euler bilinear form.
Together with results in silting theory, we have the following.
It is open for arbitrary abelian length categories $\AA$
with finitely many simple objects.

\begin{proposition}\label{M-TF brick all}
The algebra $A$ is brick finite if and only if
there exists $M \in \mod A$ such that the $M$-TF equivalence
coincides with the TF equivalence.
\end{proposition}

\begin{proof}
The ``only if'' part follows from Proposition \ref{M-TF indec}
by setting $M:=\bigoplus_{S \in \brick A}S$.

We prove the ``if'' part.
Take such $M \in \mod A$.
Then Theorem \ref{TF(M) Sigma(M)}(a) implies that
there are only finitely many $M$-TF equivalence classes.
By assumption, there exist only finitely TF equivalence classes.
Then Proposition \ref{silting TF}(b) implies that 
$\Sigma_A$ is a finite set,
so $A$ is brick finite by Proposition \ref{nonsingular}(d).
\end{proof}

Now fix $M \in \mod A$.
For $T \in \twosilt A$, 
the notion $\rmt_T M:=\rmt_\theta M$ is well-defined
by taking $\theta \in C^\circ(T)$,
because $C^\circ(T)$ is a TF equivalence class
by Proposition \ref{silting TF}.
This induces an equivalence relation on $\twosilt A$ as follows.

\begin{definition-proposition}
Let $A$ be a finite dimensional algebra and $M \in \mod A$. For $T,U\in\twosilt A$, the conditions below are equivalent. In this case, we write $T\sim_MU$.
\begin{enumerate}[\rm (a)]
\item 
$\rmt_TM=\rmt_UM$.
\item 
There exist
$\theta \in C^\circ(T)$ and $\eta \in C^\circ(U)$ such that $\theta$ and $\eta$ are $M$-TF equivalent.
\item 
Each $\theta \in C^\circ(T)$ and $\eta \in C^\circ(U)$ are $M$-TF equivalent.
\end{enumerate}
\end{definition-proposition}

\begin{proof}
(c)$\Rightarrow$(b) This is clear.

(b)$\Rightarrow$(a) 
Since $\theta$ and $\eta$ are $M$-TF equivalent,
we have $\rmt_\theta M=\rmt_\eta M$.
Thus we get $\rmt_T M=\rmt_U M$.

(a)$\Rightarrow$(c) 
By (a), we get $\rmt_\theta M=\rmt_\eta M$.
Proposition \ref{silting TF}(c) implies $\WW_\theta=\WW_\eta=\{0\}$,
so $\rmw_\theta M=\rmw_\eta M=0$.
This and $\rmt_\theta M=\rmt_\eta M$ give
$\overline{\rmt}_\theta M=\overline{\rmt}_\eta M$.
Thus we get $\rmf_\theta M=M/\overline{\rmt}_\theta M
=M/\overline{\rmt}_\eta M=\rmf_\theta M$.
\end{proof}

Then we immediately have the following result.

\begin{proposition}\label{connection to g-fan}
Let $A$ be a finite dimensional algebra and $M\in\mod A$. Then $\Sigma(M)$ is a completion of the coarsening of the $g$-fan $\Sigma_A$ with respect to the equivalence relation $\sim_M$.
\end{proposition}

\subsection{Examples}\label{Subsec example}
In this subsection, we give examples of $\Sigma(M)$. 
The first one below shows 
that the condition (ii) $\supp_\theta M=\supp_\eta M$ 
of Definition \ref{define M-TF}(a) is necessary to assure that $\Sigma(M)$ is a generalized fan.

\begin{example}\label{A2}
Let $A$ be the path algebra $k(1 \to 2)$,
and $M:=\begin{smallmatrix}1\\2\end{smallmatrix}$.
Then $\TF(M)$ has exactly seven elements in the following table.
\begin{center}
\begin{tabular}{l|ccc|c}
$M$-TF equivalence classes $E$ 
& $\rmt_E M$ & $\rmw_E M$ & $\rmf_E M$ & $\supp_E(\rmw_E M)$ \\
\hline
$E_1=\{0\}$ & $0$ & $M$ & $0$ & $\{S(2),S(1)\}$ \\
$E_2=\R_{>0}[S(2)]^*$ & $S(2)$ & $S(1)$ & $0$ & $\{S(1)\}$ \\
$E_3=\R_{>0}(-[S(1)]^*)$ & $0$ & $S(2)$ & $S(1)$ & $\{S(2)\}$ \\
$E_4=\R_{>0}([S(1)]^*-[S(2)]^*)$ & $0$ & $M$ & $0$ & $\{M\}$ \\
$E_5=E_4+E_2$ & $M$ & $0$ & $0$ & $\emptyset$ \\ 
$E_6=E_2+E_3$ & $S(2)$ & $0$ & $S(1)$ & $\emptyset$ \\ 
$E_7=E_3+E_4$ & $0$ & $0$ & $M$ & $\emptyset$ 
\end{tabular}
\end{center}
They are depicted as follows, and 
$\Sigma(M)=\{\overline{E_i} \mid i\in \{1,\ldots,7\}\}$ is 
indeed a rational fan which is finite and complete:
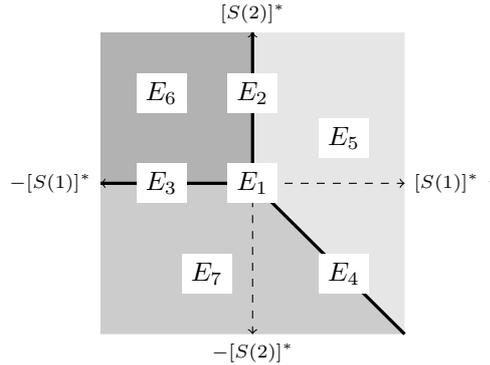
\begin{align*}
\begin{tikzpicture}[baseline=0pt,scale=1]
\node (00)[coordinate] at ( 0, 0) {};
\node (+0)[coordinate,label=  0:{$\scriptstyle{ [S(1)]^*}$}] at ( 2, 0) {};
\node (++)[coordinate]                                       at ( 2, 2) {};
\node (0+)[coordinate,label= 90:{$\scriptstyle{ [S(2)]^*}$}] at ( 0, 2) {};
\node (-+)[coordinate]                                       at (-2, 2) {};
\node (-0)[coordinate,label=180:{$\scriptstyle{-[S(1)]^*}$}] at (-2, 0) {};
\node (--)[coordinate]                                       at (-2,-2) {};
\node (0-)[coordinate,label=270:{$\scriptstyle{-[S(2)]^*}$}] at ( 0,-2) {};
\node (+-)[coordinate]                                       at ( 2,-2) {};
\fill[black!10] (00)--(+-)--(++)--(0+)--cycle;
\fill[black!30] (00)--(0+)--(-+)--(-0)--cycle;
\fill[black!20] (00)--(-0)--(--)--(+-)--cycle;
\draw[dashed,->] (00) to (+0);
\draw[dashed,->] (00) to (0+);
\draw[dashed,->] (00) to (-0);
\draw[dashed,->] (00) to (0-);
\draw[very thick] (00) to (0+);
\draw[very thick] (00) to (-0);
\draw[very thick] (00) to (+-);
\node[fill=white] (E1) at (   0,   0) {$E_1$};
\node[fill=white] (E2) at (   0, 1.2) {$E_2$};
\node[fill=white] (E3) at (-1.2,   0) {$E_3$};
\node[fill=white] (E4) at ( 1.2,-1.2) {$E_4$};
\node[fill=white] (E5) at ( 1.2, 0.6) {$E_5$};
\node[fill=white] (E6) at (-1.2, 1.2) {$E_6$};
\node[fill=white] (E7) at (-0.6,-1.2) {$E_7$};
\end{tikzpicture}.
\end{align*}

Note that $\Theta_M=E_1\sqcup E_4$ holds. Moreover
the two $M$-TF equivalence classes $E_1$ and $E_4$
are distinguished only by $\supp_\theta(\rmw_\theta M)$.
If we consider the equivalence relation on $K_0(\AA)_\R^*$
defined by the condition (i) in Definition \ref{define M-TF}(a) alone,
then the equivalence classes are $E_1 \sqcup E_4=\overline{E_4}$
and $E_2,E_3,E_5,E_6,E_7$.
The set of the closures of these six sets is 
$\Sigma(M) \setminus \{\{0\}\}$.
This is not closed under taking faces,
because $\{0\}$ is a face of $\overline{E_4}$.

Moreover $\Theta_M=\overline{E_7} \cap \overline{E_5}$
holds as in Corollary \ref{Theta_M Sigma(M)}.
\end{example}

Examples \ref{A2 wall} and \ref{A2} show that 
$\Sigma(M)$ is not necessarily given by a union of walls.

We have another example,
which implies that 
the $M$-TF equivalence for an object $M \in \AA$ which is not a brick
can be different from the $M'$-TF equivalence
for any object $M'$ which is a direct sum of bricks.
Thus it is meaningful to consider $M$-TF equivalences 
also for objects $M \in \AA$ which are not bricks.

\begin{example}
Define a finite dimensional algebra $A$ and a module $M \in \mod A$ by 
\begin{align*}
A:=k \left( \begin{tikzpicture}[baseline=(1.base),->]
\node (1) at (0, 0) {$1$};
\node (2) at (2, 0) {$2$};
\draw [thick,transform canvas={shift={( 90:0.1cm)}}] 
(1) to [edge label={\scriptsize $a$}] (2);
\draw [thick,transform canvas={shift={(270:0.1cm)}}]
(2) to [edge label={\scriptsize $b$}] (1);
\end{tikzpicture} \right)/\langle aba,bab \rangle, \quad
M:=\begin{smallmatrix}1\\2\\1\end{smallmatrix}.
\end{align*}
Then the $M$-TF equivalence classes are depicted as follows:
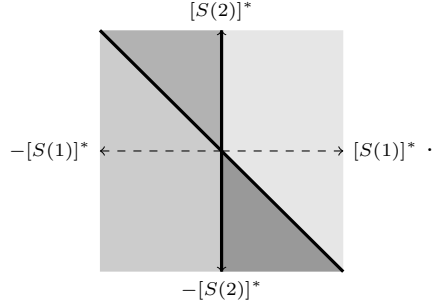
\begin{align*}
\begin{tikzpicture}[baseline=0pt,scale=0.8]
\node (00)[coordinate] at ( 0, 0) {};
\node (+0)[coordinate,label=  0:{$\scriptstyle{ [S(1)]^*}$}] at ( 2, 0) {};
\node (++)[coordinate]                                       at ( 2, 2) {};
\node (0+)[coordinate,label= 90:{$\scriptstyle{ [S(2)]^*}$}] at ( 0, 2) {};
\node (-+)[coordinate]                                       at (-2, 2) {};
\node (-0)[coordinate,label=180:{$\scriptstyle{-[S(1)]^*}$}] at (-2, 0) {};
\node (--)[coordinate]                                       at (-2,-2) {};
\node (0-)[coordinate,label=270:{$\scriptstyle{-[S(2)]^*}$}] at ( 0,-2) {};
\node (+-)[coordinate]                                       at ( 2,-2) {};
\fill[black!10] (00)--(+-)--(++)--(0+)--cycle;
\fill[black!30] (00)--(0+)--(-+)--cycle;
\fill[black!20] (00)--(0-)--(--)--(-+)--cycle;
\fill[black!40] (00)--(0-)--(+-)--cycle;
\draw[dashed,->] (00) to (+0);
\draw[dashed,->] (00) to (0+);
\draw[dashed,->] (00) to (-0);
\draw[dashed,->] (00) to (0-);
\draw[very thick] (0+) to (0-);
\draw[very thick] (-+) to (+-);
\end{tikzpicture}.
\end{align*}
There are exactly four isoclasses of bricks 
$\begin{smallmatrix}1\end{smallmatrix}$,
$\begin{smallmatrix}2\end{smallmatrix}$,
$\begin{smallmatrix}1\\2\end{smallmatrix}$ and
$\begin{smallmatrix}2\\1\end{smallmatrix}$ in $\mod A$,
and the $S$-TF equivalence classes for each brick $S$ is given as follows:
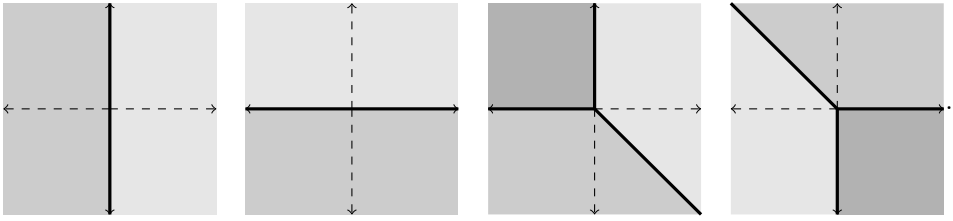
\begin{align*}
\begin{tikzpicture}[baseline=0pt,scale=0.7]
\node (00)[coordinate] at ( 0, 0) {};
\node (+0)[coordinate] at ( 2, 0) {};
\node (++)[coordinate] at ( 2, 2) {};
\node (0+)[coordinate] at ( 0, 2) {};
\node (-+)[coordinate] at (-2, 2) {};
\node (-0)[coordinate] at (-2, 0) {};
\node (--)[coordinate] at (-2,-2) {};
\node (0-)[coordinate] at ( 0,-2) {};
\node (+-)[coordinate] at ( 2,-2) {};
\fill[black!10] (++)--(+-)--(0-)--(0+)--cycle;
\fill[black!20] (-+)--(--)--(0-)--(0+)--cycle;
\draw[dashed,->] (00) to (+0);
\draw[dashed,->] (00) to (0+);
\draw[dashed,->] (00) to (-0);
\draw[dashed,->] (00) to (0-);
\draw[very thick] (0-) to (0+);
\end{tikzpicture} \quad
\begin{tikzpicture}[baseline=0pt,scale=0.7]
\node (00)[coordinate] at ( 0, 0) {};
\node (+0)[coordinate] at ( 2, 0) {};
\node (++)[coordinate] at ( 2, 2) {};
\node (0+)[coordinate] at ( 0, 2) {};
\node (-+)[coordinate] at (-2, 2) {};
\node (-0)[coordinate] at (-2, 0) {};
\node (--)[coordinate] at (-2,-2) {};
\node (0-)[coordinate] at ( 0,-2) {};
\node (+-)[coordinate] at ( 2,-2) {};
\fill[black!10] (++)--(-+)--(-0)--(+0)--cycle;
\fill[black!20] (+-)--(--)--(-0)--(+0)--cycle;
\draw[dashed,->] (00) to (+0);
\draw[dashed,->] (00) to (0+);
\draw[dashed,->] (00) to (-0);
\draw[dashed,->] (00) to (0-);
\draw[very thick] (-0) to (+0);
\end{tikzpicture} \quad
\begin{tikzpicture}[baseline=0pt,scale=0.7]
\node (00)[coordinate] at ( 0, 0) {};
\node (+0)[coordinate] at ( 2, 0) {};
\node (++)[coordinate] at ( 2, 2) {};
\node (0+)[coordinate] at ( 0, 2) {};
\node (-+)[coordinate] at (-2, 2) {};
\node (-0)[coordinate] at (-2, 0) {};
\node (--)[coordinate] at (-2,-2) {};
\node (0-)[coordinate] at ( 0,-2) {};
\node (+-)[coordinate] at ( 2,-2) {};
\fill[black!10] (00)--(+-)--(++)--(0+)--cycle;
\fill[black!30] (00)--(0+)--(-+)--(-0)--cycle;
\fill[black!20] (00)--(-0)--(--)--(+-)--cycle;
\draw[dashed,->] (00) to (+0);
\draw[dashed,->] (00) to (0+);
\draw[dashed,->] (00) to (-0);
\draw[dashed,->] (00) to (0-);
\draw[very thick] (00) to (0+);
\draw[very thick] (00) to (-0);
\draw[very thick] (00) to (+-);
\end{tikzpicture} \quad
\begin{tikzpicture}[baseline=0pt,scale=0.7]
\node (00)[coordinate] at ( 0, 0) {};
\node (+0)[coordinate] at ( 2, 0) {};
\node (++)[coordinate] at ( 2, 2) {};
\node (0+)[coordinate] at ( 0, 2) {};
\node (-+)[coordinate] at (-2, 2) {};
\node (-0)[coordinate] at (-2, 0) {};
\node (--)[coordinate] at (-2,-2) {};
\node (0-)[coordinate] at ( 0,-2) {};
\node (+-)[coordinate] at ( 2,-2) {};
\fill[black!10] (00)--(-+)--(--)--(0-)--cycle;
\fill[black!30] (00)--(0-)--(+-)--(+0)--cycle;
\fill[black!20] (00)--(+0)--(++)--(-+)--cycle;
\draw[dashed,->] (00) to (-0);
\draw[dashed,->] (00) to (0-);
\draw[dashed,->] (00) to (+0);
\draw[dashed,->] (00) to (0+);
\draw[very thick] (00) to (0-);
\draw[very thick] (00) to (+0);
\draw[very thick] (00) to (-+);
\end{tikzpicture}.
\end{align*}
Thus if $M'$ is a direct sum of bricks, 
then the $M'$-TF equivalence never coincides 
with the $M$-TF equivalence by Proposition \ref{M-TF indec}(a).
\end{example}

\subsection{Proof of the main results}\label{Subsec proof}

This subsection is devoted to proving the results in 
Subsection \ref{Subsec def M-TF}.
We begin with the following observation.

\begin{lemma}\label{t_eta M/t_theta M}
Let $M \in \AA$.
Assume that $\theta,\eta \in K_0(\AA)_\R^*$ satisfies 
$\theta \in \overline{[\eta]_M}$.
\begin{enumerate}[\rm(a)]
\item 
We have $\rmt_\eta M,\overline{\rmt}_\eta M \in \overline{\TT}_\theta$,
$\rmf_\eta M,\overline{\rmf}_\eta M \in \overline{\FF}_\theta$
and $\supp_\eta(\rmw_\eta M) \subset \WW_\theta$.
Therefore any subobject of $\rmw_\eta M$ in $\WW_\eta$ belongs to 
$\WW_\theta$.
\item 
We have a sequence
\begin{align*}
\rmt_\theta M \subset \rmt_\eta M \subset \overline{\rmt}_\eta M
\subset \overline{\rmt}_\theta M
\end{align*}
whose subfactors $\rmt_\eta M/\rmt_\theta M$, 
$\overline{\rmt}_\eta M/\rmt_\eta M$ and 
$\overline{\rmt}_\theta M/\overline{\rmt}_\eta M$ belong to $\WW_\theta$.
\end{enumerate}
\end{lemma}

\begin{proof}
(a) 
We only prove $\overline{\rmt}_\eta M \in \overline{\TT}_\theta$ 
since others can be shown similarly.
The set 
\begin{align*}
\{\theta\in K_0(\AA)_\R^* \mid \text{$\theta(X)\ge0$ 
for all factor objects $X$ of $\overline{\rmt}_\eta M$}\}
\end{align*}
is closed and contains $[\eta]_M$, 
and hence it contains $\overline{[\eta]_M}$. 
Thus the assertion follows.

(b) Since $\overline{\rmt}_\eta M \in \overline{\TT}_\theta$, 
we have $\overline{\rmt}_\eta M\subset\overline{\rmt}_\theta M$. 
Since $\overline{\rmf}_\eta M \in \overline{\FF}_\theta$,
$\overline{\rmf}_\eta M$ is 
a factor object of $\overline{\rmf}_\theta M$, 
and hence $\rmt_\theta M \subset \rmt_\eta M$.
Thus we have the inclusions in the statement.

It remains to show that the subfactors are in $\WW_\theta$.

We have $\overline{\rmt}_\eta M/\rmt_\eta M=\rmw_\eta M\in\WW_\theta$ 
by (a).

Since $\overline{\rmt}_\theta M/\overline{\rmt}_\eta M
\subset M/\overline{\rmt}_\eta M=\rmf_\eta M\in\overline{\FF}_\theta$ 
again by (a), 
we have $\overline{\rmt}_\theta M/\overline{\rmt}_\eta M
\in\overline{\FF}_\theta$. 
On the other hand, $\overline{\rmt}_\theta M/\overline{\rmt}_\eta M$
is in $\overline{\TT}_\theta$, because it is a factor object
of $\overline{\rmt}_\theta M \in \overline{\TT}_\theta$.
Thus $\overline{\rmt}_\theta M/\overline{\rmt}_\eta M
\in\overline{\TT}_\theta\cap\overline{\FF}_\theta=\WW_\theta$.

Since $\rmt_\eta M \in \overline{\TT}_\theta$ by (a), 
we have $\rmt_\eta M/\rmt_\theta M \in \overline{\TT}_\theta$. 
On the other hand, $\rmt_\eta M/\rmt_\theta M\subset M/\rmt_\theta M
=\overline{\rmf}_\theta M \in \overline{\FF}_\theta$, 
we have $\rmt_\eta M/\rmt_\theta M 
\in \overline{\FF}_\theta\cap\overline{\TT}_\theta=\WW_\theta$.
\end{proof}

We describe each $M$-TF equivalence class $E$ 
and its closure $\sigma=\overline{E}$ as follows.
For any subset $\XX \subset \AA$,
we have defined a subspace $H_\XX \subset K_0(\AA)_\R^*$ 
before Lemma \ref{Theta_M face} by
\begin{align*}
H_\XX:=\{ \theta \in K_0(\AA)_\R^* \mid 
\text{for any $X \in \XX$, $\theta(X)=0$} \}
=\{ \theta \in K_0(\AA)_\R^* \mid \theta(\R \XX)=0 \}.
\end{align*}

\begin{proposition}\label{E TWF}
Let $M \in \AA$, $E \in \TF(M)$ and $\sigma:=\overline{E} \in \Sigma(M)$.
Set $\XX:=\supp_E(\rmw_E M)$.
\begin{enumerate}[\rm (a)]
\item
We have
\begin{align*}
E=\{ \theta \in K_0(\AA)_\R^* \mid
\rmt_E M \in \TT_\theta, \ \XX \subset \simple \WW_\theta, \ 
\rmf_E M \in \FF_\theta \}.
\end{align*}
In particular, $E$ is an nonempty open convex subset of 
the subspace $H_\XX$.
\item
We have
\begin{align*}
\sigma&=\{ \theta \in K_0(\AA)_\R^* \mid
\rmt_E M \in \overline{\TT}_\theta, \ \XX \subset \WW_\theta, \ 
\rmf_E M \in \overline{\FF}_\theta \}.
\end{align*}
In particular, $\sigma$ is a rational polyhedral cone in $K_0(\AA)_\R^*$
such that $\R\sigma=H_\XX$.
\item
We have $E=\sigma^\circ$.
\end{enumerate}
\end{proposition}

\begin{proof}
(a)
The ``$\subset$'' part is clear. 
In fact, for each $\theta\in E$, 
we have $\rmt_E M=\rmt_\theta M\in\TT_\theta$, 
$\supp_E(\rmw_E M)=\supp_\theta(\rmw_\theta M) \subset \simple \WW_\theta$ 
and $\rmf_E M=\rmf_\theta M\in\FF_\theta$.

To show the ``$\supset$'' part, let $\theta$ in the right-hand side.
Then $\rmt_E M \in \TT_\theta$ implies 
$\rmt_E M \subset \rmt_\theta M$, and
$\rmf_E M \in \FF_\theta$ implies 
$\overline{\rmt}_\theta M \subset \overline{\rmt}_E M$,
so we have $\rmt_E M \subset \rmt_\theta M \subset 
\overline{\rmt}_\theta M \subset \overline{\rmt}_E M$.
Thus $\rmt_\theta M/\rmt_E M$ is 
a subobject of $\overline{\rmt}_E M/\rmt_E M=\rmw_E M$.
Since $\supp_E(\rmw_E M) \subset \simple \WW_\theta$,
we get $\rmw_E M \in \WW_\theta \subset \overline{\FF}_\theta$.
Thus we have $\rmt_\theta M/\rmt_E M \in \overline{\FF}_\theta$.
Since $\rmt_\theta M/\rmt_E M \in \TT_\theta$ holds, 
we have $\rmt_\theta M/\rmt_E M=
\TT_\theta \cap \overline{\FF}_\theta=\{0\}$;
hence $\rmt_E M=\rmt_\theta M$.
Dually, we have $\rmf_E M=\rmf_\theta M$ 
and hence $\rmw_E M=\rmw_\theta M$.
Then $\supp_E(\rmw_E M) \subset \simple \WW_\theta$ implies 
$\supp_E(\rmw_E M)=\supp_\theta(\rmw_\theta M)$.
Thus $\theta \in E$ as desired.

The equality in the statement is proved.
In particular, $E$ is the subset of $H_\XX$
given by finitely many strict linear inequalities.
Thus the last statement holds.

(b)
By Lemma \ref{t_eta M/t_theta M}(a), 
$\sigma=\overline{E}$ is contained in the right-hand side.

To show the converse, take some $\eta \in E$.
Then if $\theta$ is in the right-hand side,
for any $\epsilon \in \R_{>0}$, 
we can check $\theta':=\theta+\epsilon\eta$ satisfies 
$\rmt_E M \in \overline{\TT}_\theta \cap \TT_\eta \subset \TT_{\theta'}$,
$\supp_E(\rmw_E M) \subset \WW_\theta \cap (\simple \WW_\eta) \subset
\simple \WW_{\theta'}$ and
$\rmf_E M \in \overline{\FF}_\theta \cap \FF_\eta \subset \FF_{\theta'}$.
Thus (a) implies $\theta' \in E$ for any $\epsilon \in \R_{>0}$.
Then we have $\theta \in \overline{E}=\sigma$.

Now we have proved the equality in the statement.
The right-hand side is a rational polyhedral cone, and so is $\sigma$.
By (a), we have $\R\sigma=H_\XX$.

(c)
We show the ``$\subset$'' part first.
By (a)(b), $E$ is an open set of 
$\R \sigma=H_\XX$ with $E \subset \sigma \subset \R \sigma$.
Thus $E \subset \sigma^\circ$.

For the ``$\supset$'' part, let $\theta \in \sigma^\circ$.
Take some $\eta \in E$.
Then since $\theta \in \sigma^\circ$ and $\eta \in E \subset \R \sigma$,
there exists a sufficiently small $\epsilon \in \R_{>0}$
such that $\theta':=\theta-\epsilon \eta \in \sigma$.
We get $\theta=\theta'+\epsilon \eta \in \sigma+E$.
On the other hand, (a)(b) give $\sigma+E \subset E$.
Thus $\theta \in E$ as desired.
\end{proof}

Now we are ready to prove Theorem \ref{TF(M) Sigma(M)}.
We define a partial order $\le$ on $K_0(\AA)_\R$
so that $v \le v'$ holds if and only if 
$v'-v \in \sum_{i=1}^n \R_{\ge 0} [S(i)]$,
and $v<v'$ means $v \le v'$ and $v \ne v'$.

\begin{proof}[Proof of Theorem \ref{TF(M) Sigma(M)}]
(a) Each $\sigma\in\Sigma(M)$ is a rational polyhedral cone 
by Proposition \ref{E TWF}(b).

We prove that $\TF(M)$ and $\Sigma(M)$ are finite sets.
Consider all the linear equalities/inequalities on $\theta \in K_0(\AA)_\R^*$
of the forms $\langle \theta,v \rangle=0$, 
$\langle \theta,v \rangle>0$, 
$\langle \theta,v \rangle<0$
for some $v \in K_0(\AA)$ with $0 \le v \le [M]$.
These are a finite collection of linear equalities/inequalities 
on $\theta \in K_0(\AA)_\R^*$.
By Proposition \ref{E TWF}(a), each $E \in \TF(M)$ is determined by
some of such linear equalities/inequalities on $\theta \in K_0(\AA)_\R^*$.
Thus $\TF(M)$ is a finite set, and so is $\Sigma(M)$.

(b) The map $\Phi \colon \TF(M) \to \Sigma(M)$ with
$E \mapsto \overline{E}$ is clearly well-defined and surjective.
Proposition \ref{E TWF}(c) implies that
the map $\Psi \colon \Sigma(M) \to \TF(M)$ 
with $\sigma \mapsto \sigma^\circ$ is well-defined
and $\Psi \Phi$ is identity.
Thus $\Phi$ and $\Psi$ are bijective, and $\Phi=\Psi^{-1}$.
\end{proof}

Then the $M$-TF equivalence for each $M$ can be characterized as follows.

\begin{lemma}\label{closure both}
Let $M \in \AA$ and $\theta,\eta \in K_0(\AA)_\R^*$.
Then the following conditions are equivalent.
\begin{enumerate}[\rm (a)]
\item 
$\theta$ and $\eta$ are $M$-TF equivalent.
\item
Both $\theta \in \overline{[\eta]_M}$ and 
$\eta \in \overline{[\theta]_M}$ hold.
\end{enumerate}
\end{lemma}

\begin{proof}
(a)$\Rightarrow$(b) is obvious.

We show (b)$\Rightarrow$(a).
Since $\theta \in \overline{[\eta]_M}$ and 
$\eta \in \overline{[\theta]_M}$,
Lemma \ref{t_eta M/t_theta M}(b) implies $\rmt_\theta M=\rmt_\eta M$.
Similarly $\rmf_\theta M=\rmf_\eta M$ holds;
hence $\rmw_\theta M=\rmw_\eta M$.

By Proposition \ref{E TWF}(a), it remains to show
$\supp_\theta(\rmw_\theta M) \subset \simple \WW_\eta$.
If $X \in \supp_\theta(\rmw_\theta M)$,
then $\eta \in \overline{[\theta]_M}$ implies $X \in \WW_\eta$
by Lemma \ref{t_eta M/t_theta M}(a).
Since $X \ne 0$, take a simple subobject $Y$ of $X$ in $\WW_\eta$.
Then $\theta \in \overline{[\eta]_M}$ implies $Y \in \WW_\theta$
again by Lemma \ref{t_eta M/t_theta M}(a),
so $Y$ is a subobject of $X$ in $\WW_\theta$.
Since $X \in \supp_\theta(\rmw_\theta M)$, 
we have $X=Y \in \simple \WW_\eta$. 
Therefore $\supp_\theta(\rmw_\theta M) \subset \simple \WW_\eta$ 
as desired.
\end{proof}

Next we aim to rewrite the $M$-TF equivalence
in terms of the following set of some certain subobjects of $M$.

\begin{definition-proposition}
Let $M \in \AA$, and $\theta \in K_0(\AA)_\R^*$.
We define $\rmt(\theta,M)$ as the set of subobjects $L$ of $M$
satisfying the following equivalent conditions.
\begin{enumerate}[\rm (a)]
\item 
$\rmt_\theta M \subset L$ and 
$L/\rmt_\theta M \in \WW_\theta$.
\item 
$\rmt_\theta M \subset L$ and 
$L/\rmt_\theta M$ is a subobject of $\rmw_\theta M$ in $\WW_\theta$.
\item 
$L \subset \overline{\rmt}_\theta M$ and 
$\overline{\rmt}_\theta M/L \in \WW_\theta$.
\item 
$L \subset \overline{\rmt}_\theta M$ and 
$\overline{\rmt}_\theta M/L$ is a factor object of 
$\rmw_\theta M$ in $\WW_\theta$.
\end{enumerate}
In particular, we have $\rmt_\theta M,\overline{\rmt}_\theta M \in \rmt(\theta,M)$ and $\rmt(\theta,M) \subset \overline{\TT}_\theta$.
\end{definition-proposition}

\begin{proof}
(b)$\Rightarrow$(a) and (d)$\Rightarrow$(c) are clear.

(a)$\Rightarrow$((b) and (d))
Since $\rmt_\theta M \in \TT_\theta$
and $L/\rmt_\theta M \in \WW_\theta$, 
we get $L \in \overline{\TT}_\theta$.
Thus we have $\rmt_\theta M \subset L \subset \overline{\rmt}_\theta M$.
This and $L/\rmt_\theta M \in \WW_\theta$ imply that
$L/\rmt_\theta M$ is a subobject of $\rmw_\theta M$ in $\WW_\theta$.
Thus $\overline{\rmt}_\theta M/L$ is a factor object of 
$\rmw_\theta M$ in $\WW_\theta$.
Therefore we get (b) and (d).

Similarly, (c)$\Rightarrow$((b) and (d)) is true.
\end{proof}

We remark that $\rmt(\theta,M)$ coincides with $\LL(\theta,M)$ 
defined in \cite[Section 4]{Fei2},
which follows from Lemma \ref{N(M)_theta} later.

By using $\rmt(\theta,M)$, 
we can characterize elements in $\Sigma(M)$ and $\TF(M)$ as follows.

\begin{lemma}\label{L(theta,M) incl}
Let $M \in \AA$, and $\theta,\eta \in K_0(\AA)_\R^*$.
\begin{enumerate}[\rm (a)]
\item
The following conditions are equivalent.
\begin{enumerate}[\rm (i)]
\item
$\theta \in \overline{[\eta]_M}$.
\item
$\rmt(\eta,M) \subset \rmt(\theta,M)$.
\end{enumerate}
\item
The following conditions are equivalent.
\begin{enumerate}[\rm (i)]
\item
$\theta$ and $\eta$ are $M$-TF equivalent.
\item
$\rmt(\eta,M)=\rmt(\theta,M)$.
\end{enumerate}
\end{enumerate}
\end{lemma}

\begin{proof}
(a)
(i)$\Rightarrow$(ii)
Let $L \in \rmt(\eta,M)$.
We have $\rmt_\eta M \subset L$, 
and $L/\rmt_\eta M$ is a subobject of $\rmw_\eta M$ in $\WW_\eta$.
Then by Lemma \ref{t_eta M/t_theta M}(b)
and $\theta \in \overline{[\eta]_M}$, 
we have $\rmt_\theta M \subset \rmt_\eta M \subset L$
and $\rmt_\eta M/\rmt_\theta M \in \WW_\theta$.
Moreover, since
$L/\rmt_\eta M$ is a subobject of $\rmw_\eta M$ in $\WW_\eta$,
Lemma \ref{t_eta M/t_theta M}(a) and
$\theta \in \overline{[\eta]_M}$ imply
$L/\rmt_\eta M \in \WW_\theta$.
Therefore $\rmt_\theta M \subset L$ and 
$L/\rmt_\theta M \in \WW_\theta$ hold.
Thus we get $L \in \rmt(\theta,M)$.

(ii)$\Rightarrow$(i)
By Proposition \ref{E TWF}(b),
it suffices to show
(1) $\rmt_\eta M \in \overline{\TT}_\theta$,
(2) $\supp_\eta(\rmw_\eta M) \in \WW_\theta$ and
(3) $\rmf_\eta M \in \overline{\FF}_\theta$.

(1)
Since $\rmt_\eta M \in \rmt(\eta,M)$,
(ii) implies $\rmt_\eta M \in \rmt(\theta,M) \in \overline{\TT}_\theta$.

(2)
Considering a composition series of $\rmw_\eta M$ in $\WW_\eta$,
the desired condition $\supp_E(\rmw_E M) \subset \WW_\theta$ holds
if any subobject $L' \subset \rmw_\eta M$ in $\WW_\eta$ 
belongs to $\WW_\theta$.
Let $L'$ be a subobject of $\rmw_\eta M$ in $\WW_\eta$.
Since $L'$ is a subobject of $\overline{\rmf}_\eta M=M/\rmt_\eta M$,
take a subobject $L \subset M$ such that 
$\rmt_\eta M \subset L$ and $L/\rmt_\eta M \simeq L' \in \WW_\eta$.
Then we have $L \in \rmt(\eta,M)$, and then (ii) implies $L \in \rmt(\theta,M)$.
Thus we get $L/\rmt_\theta M \in \WW_\theta$.
Since $\rmt_\eta M \in \rmt(\eta,M) \subset \rmt(\theta,M)$ by (ii), 
we have $\rmt_\theta M \subset \rmt_\eta M$ and 
$\rmt_\eta M/\rmt_\theta M \in \WW_\theta$.
Thus $L/\rmt_\eta M=(L/\rmt_\theta M)/(\rmt_\eta M/\rmt_\theta M) 
\in \WW_\theta$; hence $L' \in \WW_\theta$.

(3) 
We have $\overline{\rmt}_\eta M \in \rmt(\eta,M)$.
Then by (ii), we have $\overline{\rmt}_\eta M \in \rmt(\theta,M)$.
Thus $\overline{\rmt}_\eta M \subset \overline{\rmt}_\theta M$ and
$\overline{\rmt}_\theta M/\overline{\rmt}_\eta M \in \WW_\theta$ hold.
This and $M/\overline{\rmt}_\theta M=\rmf_\theta M \in \FF_\theta$ give
$\rmf_\eta M=M/\overline{\rmt}_\eta M \in \overline{\FF}_\theta$.

(b) follows from (a) and Lemma \ref{closure both}.
\end{proof}

Next we explain that the set $\rmt(\theta,M)$ 
for each $\theta \in K_0(\AA)_\R^*$
is strongly related to the Newton polytope $\rmN(M)$.
Recall that we have defined a partial order $\le$ on $K_0(\AA)_\R$
so that $v \le v'$ holds if and only if 
$v'-v \in \sum_{i=1}^n \R_{\ge 0} [S(i)]$.
The notation $v<v'$ means $v \le v'$ and $v \ne v'$.
The following lemma on
\begin{align*}
\rmN(M)_\theta=\{v \in \rmN(M) \mid \theta(v)=\max \theta(\rmN(M))\}
\end{align*}
is crucial; 
see also \cite[Subsection 1.4]{BKT} and 
\cite[Theorem 4.4, Proposition 8.6]{Fei2}.

\begin{lemma}\label{N(M)_theta}
\cite[Lemmas 5.19, 5.20]{AHIKM}
Let $M \in \AA$ and $\theta \in K_0(\AA)_\R^*$.
\begin{enumerate}[\rm (a)]
\item
Let $L$ be a subobject of $M$.
Then $L \in \rmt(\theta,M)$ is equivalent to $[L] \in \rmN(M)_\theta$.
In particular, for each $L\in\rmt(\theta,M)$, we have $\theta(L)=\max\theta(\rmN(M))$.
\item
In $K_0(\AA)_\R$, we have 
\begin{align*}
\rmN(M)_\theta=[\rmt_\theta M]+\conv\{ [X] \mid 
\text{$X$ is a subobject of $\rmw_\theta M$ in $\WW_\theta$}\}
=\conv\{ [L] \mid L \in \rmt(\theta,M)\}.
\end{align*}
\item
Under the partial order on $K_0(\AA)_\R$ above, we have
\begin{align*}
\min \rmN(M)_\theta=[\rmt_\theta M], \quad 
\max \rmN(M)_\theta=[\overline{\rmt}_\theta M].
\end{align*}
In particular, each edge $[u,v]$ of $\rmN(M)$ satisfies $u<v$ and $v<u$.
\end{enumerate}
\end{lemma}

This and Lemma \ref{L(theta,M) incl}(a) imply the following.

\begin{lemma}\label{M-TF N(M)_theta incl}
Let $\theta,\eta \in K_0(\AA)_\R^*$.
Then the following conditions are equivalent.
\begin{enumerate}[\rm (a)]
\item
$\theta \in \overline{[\eta]_M}$.
\item
$\rmt(\eta,M) \subset \rmt(\theta,M)$.
\item
$\rmN(M)_\eta \subset \rmN(M)_\theta$.
\end{enumerate}
In this case, we have 
\begin{align*}
&\theta(\rmt_\theta M)=\theta(\overline{\rmt}_\theta M)=
\theta(\rmt_\eta M)=\theta(\overline{\rmt}_\eta M)
=\max \theta(\rmN(M)),\\
&\theta(\rmf_\theta M)=\theta(\overline{\rmf}_\theta M)=
\theta(\rmf_\eta M)=\theta(\overline{\rmf}_\eta M)
=\theta(M)-\max \theta(\rmN(M)).
\end{align*}
\end{lemma}

\begin{proof}
(a)$\Leftrightarrow$(b) are just Lemma \ref{L(theta,M) incl}(a).

(b)$\Rightarrow$(c) follows from Lemma \ref{N(M)_theta}(b),
and (c)$\Rightarrow$(b) follows from Lemma \ref{N(M)_theta}(a).

The last part follows from Lemma \ref{N(M)_theta}(a).
\end{proof}

Thus we have the following characterization of the $M$-TF equivalence.

\begin{lemma}\label{M-TF N(M)_theta}
Fix $M \in \AA$. 
\begin{enumerate}[\rm (a)]
\item
For any $\theta,\eta \in K_0(\AA)_\R^*$, 
the following conditions are equivalent.
\begin{enumerate}[\rm (i)]
\item
$\theta$ and $\eta$ are $M$-TF equivalent.
\item
$\rmt(\theta,M)=\rmt(\eta,M)$.
\item
$\rmN(M)_\theta=\rmN(M)_\eta$.
\end{enumerate}
\item
Let $E \in \TF(M)$. 
Then $\rmN(M)_E \in \Face \rmN(M)$ is well-defined
by setting $\rmN(M)_E:=\rmN(M)_\eta$ for any $\eta \in E$.
Moreover we have $E=(\sigma_{\rmN(M)_E})^\circ$ and 
$\overline{E}=\sigma_{\rmN(M)_E}$.
\end{enumerate}
\end{lemma}

\begin{proof}
(a) follows from Lemmas \ref{closure both} and \ref{M-TF N(M)_theta incl}.

(b)
By (a), $\rmN(M)_E$ is well-defined.

Then Definition-Proposition \ref{Sigma(P)}(d) implies that,
for any $\theta \in E$, we have $\theta \in (\sigma_{\rmN(M)_E})^\circ$.
Thus we get $E \subset (\sigma_{\rmN(M)_E})^\circ$.

Conversely, let $\theta \in (\sigma_{\rmN(M)_E})^\circ$.
By the uniqueness in Definition-Proposition \ref{Sigma(P)}(d),
we have $\sigma_{\rmN(M)_\theta}=\sigma_{\rmN(M)_E}$.
Thus $\rmN(M)_\theta=\rmN(M)_E$ holds 
by Definition-Proposition \ref{Sigma(P)}(c).
Then (a) implies $\theta \in E$.
Therefore $(\sigma_{\rmN(M)_E})^\circ \subset E$ also holds.

Thus $E=(\sigma_{\rmN(M)_E})^\circ$; 
hence $\overline{E}=\sigma_{\rmN(M)_E}$.
\end{proof}

Now we are able to prove the main results,
Theorem \ref{Sigma(M)=Sigma(N(M))}
and Corollaries \ref{M-TF decompose} and \ref{Theta_M Sigma(M)}.

\begin{proof}[Proof of Theorem \ref{Sigma(M)=Sigma(N(M))}]
(a) 
Definition-Proposition \ref{Sigma(P)}(b) gives 
a surjection $K_0(\AA)_\R^* \to \Sigma(\rmN(M))$
defined by $\theta \mapsto \sigma_{\rmN(M)_\theta}$.
This induces a bijection $\TF(M) \to \Sigma(\rmN(M))$
with $E \mapsto \sigma_{\rmN(M)_E}$ by Lemma \ref{M-TF N(M)_theta}(b).
On the other hand, we have
a bijection $\TF(M) \to \Sigma(M)$ in Theorem \ref{TF(M) Sigma(M)}
given by $E \mapsto \overline{E}$.
These bijections coincide again by Lemma \ref{M-TF N(M)_theta}(b),
so we have $\Sigma(M)=\Sigma(\rmN(M))$.

Now $\Sigma(M)$ is a rational generalized fan 
which is finite and complete 
by Definition-Proposition \ref{Sigma(P)}(b).

(b) follows from (a) and Definition-Proposition \ref{Sigma(P)}(c).
\end{proof}

\begin{proof}[Proof of Corollary \ref{M-TF decompose}]
(a) and (b) follow from 
that $\Sigma(M)$ is a generalized fan 
(by Theorem \ref{Sigma(M)=Sigma(N(M))})
and from that $\sigma^\circ \in \TF(M)$ holds 
for each $\sigma \in \Sigma(M)$ (by Theorem \ref{TF(M) Sigma(M)}).

(c)
Consider the $\R$-vector subspace 
$V:=H_{\supp_{E'}(\rmw_{E'} M)}$.
The desired equality is $\sigma'=\sigma \cap V$.
Since $\sigma' \in \Face \sigma$, 
it suffices to show $\R \sigma'=V$.
By Theorem \ref{TF(M) Sigma(M)}, $\sigma'=\overline{E'}$,
so Proposition \ref{E TWF}(b) gives $\R \sigma'=V$.
\end{proof}

\begin{proof}[Proof of Corollary \ref{Theta_M Sigma(M)}]
(a)
For the vertex $\{0\} \in \Face_0 \rmN(M)$,
by Theorem \ref{Sigma(M)=Sigma(N(M))},
the corresponding element in $\Sigma_n(M)$ is
$\{ \theta \in K_0(\AA)_\R^* \mid \max \theta(\rmN(M))=0\}$,
which is clearly $\sigma_-$.
Similarly, the vertex $\{[M]\}$ corresponds to $\sigma_+$.

(b)
Since $\WW_\theta=\overline{\FF}_\theta \cap \overline{\TT}_\theta$
for any $\theta \in K_0(\AA)_\R^*$,
we have $\Theta_M=\sigma_- \cap \sigma_+$.
Moreover (a) and Proposition \ref{Face_d N(M) bij} imply 
$\sigma_- \cap \sigma_+=\sigma_F \in \Sigma(M)$.
Thus Theorem \ref{TF(M) Sigma(M)}(b) gives $\Theta_M^\circ \in \TF(M)$.

(c) follows from (b) and Theorem \ref{Sigma(M)=Sigma(N(M))}(b).

(d) and (e) follow from (a) and Corollary \ref{M-TF decompose}(a)(b).
\end{proof}

We end this section with the following result on the smallest element 
of $\Sigma(M)$.

\begin{proposition}\label{smallest}
Let $M \in \AA$.
Set 
\begin{align*}
I:=\{i \in \{1,\ldots,n\} \mid 
\text{$S(i)$ is not a composition factor of $M$ in $\AA$}\}.
\end{align*}
Then the vector subspace $\sum_{i \in I}\R [S(i)]^* \subset K_0(\AA)_\R^*$
is the smallest element of $\Sigma(M)$,
and it is the $M$-TF equivalence class $[0]_M$ of $0 \in K_0(\AA)_\R^*$.
\end{proposition}

\begin{proof}
Clearly, $M \in \WW_0$.
Since $\WW_0=\AA$, we have $\supp_0 (\rmw_0 M)=\{S(i) \mid i \notin I\}$.
Then Proposition \ref{E TWF}(a) gives $[0]_M=\sum_{i \in I}\R [S(i)]^*$.
This is already closed, so $\sum_{i \in I}\R [S(i)]^* \in \Sigma(M)$.

Now let $\sigma \in \Sigma(M)$.
Then $0 \in \sigma$ holds,
so take the unique face $\tau \in \Face \sigma$ such that
$0 \in \tau^\circ$.
By Corollary \ref{M-TF decompose}(a), we get $\tau^\circ \in \TF(M)$.
Thus $\tau^\circ=[0]_M$ holds.
We have $\sum_{i \in I}\R [S(i)]^*=[0]_M=\tau^\circ \subset \sigma$
for any $\sigma \in \Sigma(M)$.
Thus $\sum_{i \in I}\R [S(i)]^*$ is the smallest element of $\Sigma(M)$,
and it is the $M$-TF equivalence class $[0]_M$.
\end{proof}

\section{Maximal cones of $\Sigma(M)$}\label{Sec max}

For each $M\in\AA$, we constructed bijections $\Sigma(M) \leftrightarrow \TF(M)$ 
and a generalized fan $\Sigma(M)$ 
which is finite and complete in $K_0(\AA)_\R^* \simeq \R^n$; 
see Theorem \ref{Sigma(M)=Sigma(N(M))}.
Thus, in this section, we use the notations like 
\begin{align*}
\rmt_\sigma M:=\rmt_{\sigma^\circ} M\ 
\text{for each $\sigma \in \Sigma(M)$.}
\end{align*}
We set $\TF_i(M)$ and $\Sigma_i(M)$ as the subsets of $\TF(M)$ and $\Sigma(M)$
consisting of $i$-dimensional elements for each $i\in \{0,\ldots,n\}$.
In this section, we mainly consider cones in $\Sigma_n(M)$, 
which are precisely the maximal cones in $\Sigma(M)$.
We begin with the following characterization of $\TF_n(M)$ and $\Sigma_n(M)$.

\begin{lemma}\label{TF Sigma max}
Let $M \in \AA$, $E \in \TF(M)$ and $\sigma:=\overline{E} \in \Sigma(M)$.
Then the following conditions are equivalent.
\begin{enumerate}[\rm (a)]
\item
$E \in \TF_n(M)$; or equivalently, $E$ is open in $K_0(\AA)_\R^*$.
\item
$\sigma \in \Sigma_n(M)$; or equivalently, $\sigma$ is maximal in $\Sigma(M)$.
\item
$\rmw_E M=0$; or equivalently, $\rmt_E M=\overline{\rmt}_E M$.
\end{enumerate}
In particular, we have a bijection $\TF_n(M) \to \Sigma_n(M)$ given by 
$E \mapsto \overline{E}$.
\end{lemma}

\begin{proof}
By Proposition \ref{E TWF}(a), $E\in\TF_n(M)$ if and only if $E$ is open in $K_0(\AA)_\R^*$. These conditions are equivalent to (c) again by Proposition \ref{E TWF}(a).

Since $\sigma$ is a polyhedral cone by Proposition \ref{E TWF}(b),
(a) and (b) are equivalent by Proposition \ref{E TWF}(c).

Then the last statement follows from Theorem \ref{TF(M) Sigma(M)}(a). 
\end{proof}

For any $E \in \TF_n(M)$ with $\sigma:=\overline{E}$,
Proposition \ref{E TWF}(a) and Lemma \ref{TF Sigma max} give
\begin{align*}
\sigma^\circ=E&=\{\theta \in K_0(\AA)_\R^* \mid
\rmt_\sigma M \in \TT_\theta, \ \rmf_\sigma M \in \FF_\theta\},\\
\partial \sigma=\sigma \setminus E&=
\{\theta \in \sigma \mid
\text{$\rmt_\sigma M \notin \TT_\theta$ or 
$\rmf_\sigma M \notin \FF_\theta$}\}.
\end{align*}

Motivated by this, we define the following subsets of
the boundary $\partial \sigma$.

\begin{definition}\label{def partial^+}
Let $M \in \AA$. For each $\sigma \in \Sigma_n(M)$, we set
\begin{align*}
\partial^+ \sigma
:=\{\theta \in \sigma \mid \rmt_\sigma M \notin \TT_\theta\}, \quad
\partial^- \sigma
:=\{\theta \in \sigma \mid \rmf_\sigma M \notin \FF_\theta\}.
\end{align*}    
\end{definition}

These satisfy the following properties.

\begin{proposition}\label{union of faces}
Let $M \in \AA$ and $\sigma \in \Sigma_n(M)$.
\begin{enumerate}[\rm (a)]
\item 
We obtain $\partial \sigma=\partial^+ \sigma \cup \partial^- \sigma$.
\item 
We have
\begin{align*}
\partial^+\sigma&=\{\theta \in \sigma \mid
\text{there exists a nonzero factor object $X$ 
of $\rmt_\sigma M$ with $\theta(X)=0$}\},\\
\partial^-\sigma&=\{\theta \in \sigma \mid
\text{there exists a nonzero subobject $X$ 
of $\rmf_\sigma M$ with $\theta(X)=0$}\}.
\end{align*}
\item 
The sets $\partial^+ \sigma$ and $\partial^- \sigma$
are unions of faces of $\sigma$.
\end{enumerate}
\end{proposition}

\begin{proof}
(a) follows from the observation
above Definition \ref{def partial^+}.

(b)
We show the first equality, because the other is dual.
The right-hand side is contained in the left-hand side
by definition.

To show the converse, let $\theta \in \partial^+ \sigma$.
Since $\rmt_\sigma M \notin \TT_\theta$,
we take a nonzero factor object $X$ of $\rmt_\sigma M$
such that $\theta(X) \le 0$.
On the other hand, 
by $\theta \in \partial^+ \sigma \subset \sigma=
\overline{\sigma^\circ}$
and Lemma \ref{t_eta M/t_theta M}(a),
we have $\rmt_\sigma M \in \overline{\TT}_\theta$,
so in particular, $\theta(X) \ge 0$ holds.
Thus we have $\theta(X)=0$ as desired.

(c)
For each nonzero factor object $X$ of $\rmt_\sigma M$,
the set $\{\theta \in \sigma \mid \theta(X)=0\}$
is a face of $\sigma$,
because any $\theta \in \sigma$ satisfies $\theta(X) \ge 0$.
Thus $\partial^+ \sigma$ is a union of faces of $\sigma$ by (b),
and so is $\partial^- \sigma$.
\end{proof}

As usual, faces of of codimension one are called \emph{facets}.
For $\sigma \in \Sigma_n(M)$, 
we write $\Facet \sigma$ for the set of facets of $\sigma$,
and 
\begin{align*}
\Facet^\pm \sigma:=\{\tau \in \Facet \sigma
\mid \tau \subset \partial^\pm \sigma\}.
\end{align*}

The main result of this section below is 
on the purity of $\partial^\pm \sigma$.
On the partial order $\le$ on $K_0(\AA)_\R$,
see before Lemma \ref{N(M)_theta}.

\begin{theorem}\label{facet +-}
For $M \in \AA$, let $\sigma \in \Sigma_n(M)$ and $\{v\}\in\Face_0 \rmN(M)$ the corresponding vertex.
\begin{enumerate}[\rm (a)]
\item 
$\Facet \sigma=\Facet^+ \sigma \sqcup \Facet^- \sigma$.
\item
We have $\partial^+\sigma=\bigcup_{\tau\in\Facet^+\sigma}\tau$ and $\partial^-\sigma=\bigcup_{\tau\in\Facet^-\sigma}\tau$.
\item Let $\sigma'\in\Sigma_n(M)$ such that $\tau:=\sigma\cap\sigma'\in\Facet\sigma$, and $\{v'\}\in\Face_0 \rmN(M)$ the corresponding vertex to $\sigma'$.
Then precisely one of the following statements holds.
\begin{enumerate}[\rm(i)]
\item $v>v'$ and $\tau\in\Facet^+\sigma$.
\item $v<v'$ and $\tau\in\Facet^-\sigma$.
\end{enumerate}
\end{enumerate}
\end{theorem}

For our purpose, 
the following property of the vertices of $\rmN(M)$ is important.
By Lemma \ref{N(M)_theta}(c), 
each $F \in \rmN(M)$ has the smallest element, 
which is denoted by $\min F$.

\begin{lemma}\label{side +-}
Let $M \in \AA$. 
Assume that $F \in \Face \rmN(M)$ and $\{v\} \in \Face_0 \rmN(M)$
satisfy $v \in F$ and $v \ne \min F$.
Then there exists a vertex $u \in F$ adjacent to $v$ and $u<v$.
\end{lemma}

\begin{proof}
Take the subset $X \subset K_0(\AA)_\R$ 
consisting of $v$ and all vertices $u \in F$ adjacent to $v$.
By Lemma \ref{N(M)_theta}(c), 
each $u \in X \setminus \{v\}$ satisfies $v<u$ and $u<v$.
As a subset of $F$, the convex hull $\conv X$ is a neighborhood of $v$ in $F$,
so $v \ne \min F$ implies 
the existence of $u \in X \setminus \{v\}$ such that $u<v$.
Thus we get the assertion. 
\end{proof}

Thus we have the following result.
For any $\sigma \in \Sigma(M)$, 
we set $\rmN(M)_\sigma:=\rmN(M)_{\sigma^\circ}$,
which is well-defined by Lemma \ref{M-TF N(M)_theta}(b).

\begin{lemma}\label{facet side}
For $M \in \AA$,
let $\sigma,\sigma' \in \Sigma_n(M)$,
and $\{v\},\{v'\} \in \Face_0 \rmN(M)$ 
the correspoinding vertices.
Then the following conditions are equivalent.
\begin{enumerate}[\rm (a)]
\item
The intersection $\sigma \cap \sigma'$ is 
in $\Sigma_{n-1}(M)$.
\item
The line segment $[v,v']$ is an edge of $\rmN(M)$.
\end{enumerate}
In this case, set 
$\tau:=\sigma \cap \sigma' \in \Sigma_{n-1}(M)$. 
Then precisely one of the following statements holds.
\begin{enumerate}[\rm (i)]
\item
$\tau=\partial^+ \sigma \cap \partial^- \sigma'$,
$\rmt_{\sigma'} M=\rmt_\tau M \subsetneq 
\overline{\rmt}_\tau M=\rmt_\sigma M$ 
and $v>v'$.
\item
$\tau=\partial^- \sigma \cap \partial^+ \sigma'$,
$\rmt_\sigma M=\rmt_\tau M \subsetneq 
\overline{\rmt}_\tau M=\rmt_{\sigma'} M$ 
and $v<v'$.
\end{enumerate}
\end{lemma}

\begin{proof}
By Proposition \ref{Face_d N(M) bij},
(a) and (b) are equivalent,
and in this case, $\rmN(M)_\tau=[v,v']$ holds.
On the other hand, since $\tau$ is an edge of $\rmN(M)$,
we have $\rmN(M)_\tau=[[\rmt_\tau M],[\overline{\rmt}_\tau M]]$
with $[\rmt_\tau M]<[\overline{\rmt}_\tau M]$
by Lemma \ref{N(M)_theta}(c).
Thus, the pair $(v,v')$ is 
$([\rmt_\tau M],[\overline{\rmt}_\tau M])$ or 
$([\overline{\rmt}_\tau M],[\rmt_\tau M])$.

In the rest, we assume the former, and show (ii).
First, $v<v'$ is already shown above.
Moreover $\tau \subset \sigma$ implies 
$\rmt_\tau M \subset \rmt_\sigma M$ by Lemma \ref{t_eta M/t_theta M}(b).
This and $[\rmt_\sigma M]=v=[\rmt_\tau M]$ give 
$\rmt_\tau M=\rmt_\sigma M$.
Similarly, by $\tau \subset \sigma'$,
we get $\overline{\rmt}_\tau M=\overline{\rmt}_{\sigma'} M$,
which is $\rmt_{\sigma'} M$ by Lemma \ref{TF Sigma max}.
Consequently, we obtain $\rmt_\sigma M=\rmt_\tau M \subset 
\overline{\rmt}_\tau M=\rmt_{\sigma'} M$.
The inclusion must be proper, since $v<v'$.
Thus $\rmt_\tau M=\rmt_\sigma M \subsetneq \rmt_{\sigma'} M$ hold,
so we get $\tau^\circ \subset \partial^+ \sigma'$ and hence
$\tau \subset \partial^+ \sigma'$.
Similarly, we have $\tau \subset \partial^- \sigma$.
Thus $\tau \subset \partial^- \sigma \cap \partial^+ \sigma'$ follows.
The converse $\partial^- \sigma \cap \partial^+ \sigma' \subset \tau$
clearly holds, since $\tau=\sigma \cap \sigma'$.
Therefore $\tau=\partial^- \sigma \cap \partial^+ \sigma'$.
\end{proof}

Then Theorem \ref{facet +-} can be proved as follows.

\begin{proof}[Proof of Theorem \ref{facet +-}]
Set $E:=\sigma^\circ$.
By Lemma \ref{TF Sigma max}, $E \in \TF_n(M)$.

(a)
We first show $\Facet \sigma=\Facet^+ \sigma \cup \Facet^- \sigma$.
The ``$\supset$'' part is clear, so we check the ``$\subset$'' part.
Let $\tau \in \Facet \sigma$.
Then $\tau \subset \partial \sigma$ holds,
so Proposition \ref{union of faces}(a) implies 
that $\tau \subset \partial^+ \sigma \cup \partial^- \sigma$.
Since $\tau$ is a face of $\sigma$,
Proposition \ref{union of faces}(c) implies that
$\tau$ is contained in at least one of 
$\partial^+ \sigma$ or $\partial^- \sigma$.
Thus $\Facet \sigma \subset \Facet^+ \sigma \cup \Facet^- \sigma$ holds.

It remains to prove $\Facet^+ \sigma \cap \Facet^- \sigma = \emptyset$.
Let $\tau \in \Facet^+ \sigma \cap \Facet^- \sigma$.
Since $\Sigma(M)$ is finite and complete
and since $\tau$ is a facet of $\sigma$,
we take $\sigma' \in \Sigma_n(M)$ 
such that $\tau=\sigma \cap \sigma'$.
Then Lemma \ref{facet side} implies $\rmt_\tau M=\rmt_{\sigma} M$ or 
$\overline{\rmt}_\tau M=\overline{\rmt}_\sigma M$.
In the former case, $\rmt_\sigma M=\rmt_\tau M \in \TT_\theta$ 
holds for any $\theta \in \tau^\circ$,
so we get $\tau \not \subset \partial^+ \sigma$, a contradiction.
In the latter case, we get $\rmf_\sigma M=\rmf_\tau M \in \FF_\theta$
for any $\theta \in \tau^\circ$,
which gives another contradiction $\tau \not \subset \partial^- \sigma$.

(b)
We only show the assertion for $\partial^+ \sigma$.
By Proposition \ref{union of faces}(c), 
$\partial^+ \sigma$ is a union of faces of $\sigma$.
Let $\rho$ be a face of $\sigma$ contained in $\partial^+ \sigma$.
It suffices to show the existence of $\tau \in \Facet^+ \sigma$
such that $\rho \subset \tau$.

Take the face $F$ and the vertex $\{v\}$ of $\rmN(M)$
corresponding to $\rho$ and $\sigma$ 
in Theorem \ref{Sigma(M)=Sigma(N(M))}(b), respectively.
Then $v \in F$ holds by $\rho \subset \sigma$. 

Since $\rho \subset \sigma$, 
we get $\rmt_\rho M \subset \rmt_\sigma M$ 
by Lemma \ref{t_eta M/t_theta M}(b).
On the other hand, $\rho \subset \partial^+ \sigma$ gives
$\rmt_\sigma M \notin \TT_\theta$ for any $\theta \in \rho$.
These imply $\rmt_\rho M \subsetneq \rmt_\sigma M$.
This and Lemma \ref{N(M)_theta}(c) give
$\min F=[\rmt_\rho M]<[\rmt_\sigma M]=v$,
because $F=\rmN(M)_\rho$ and $\{v\}=\rmN(M)_\sigma$ by definition.

We have obtained $\min F<v \in F$, 
so Lemma \ref{side +-} tells us that 
there exists a vertex $v' \in F$ such that 
$[v',v]$ is an edge of $F$ and $v'<v$.
Taking $\sigma' \in \Sigma_n(M)$ corresponding to the vertex $\{v'\}$,
we get $\tau:=\sigma \cap \sigma' \in \Sigma_{n-1}(M)$
and $\tau \subset \partial^+ \sigma$ by Lemma \ref{facet side}.
Thus $\tau \in \Facet^+ \sigma$ holds.
Since $\tau=\sigma \cap \sigma'$ corresponds to $[v',v]$
by Proposition \ref{Face_d N(M) bij},
the inclusion $[v',v] \subset F$ implies $\rho \subset \tau$.
Therefore $\tau \in \Facet^+ \sigma$ and 
$\rho \subset \tau$ hold as desired.

(c)
is immediate from Lemma \ref{facet side}.
\end{proof}

Motivated by Theorem \ref{facet +-} and Lemma \ref{facet side},
we introduce the following notions.

\begin{definition}
Let $M \in \AA$.
We define paths in $\Sigma(M)$ and $\rmN(M)$ as follows.
\begin{enumerate}[\rm (a)]
\item
A \emph{path} in $\Sigma(M)$ is a sequence
$(\sigma_0,\sigma_1,\ldots,\sigma_\ell)$ such that
$\sigma_i \in \Sigma_n(M)$ for each $i \in \{0,\ldots,\ell\}$ and
$\sigma_{i-1} \cap \sigma_i \in \Sigma_{n-1}(M)$
for each $i \in \{1,\ldots,\ell\}$.
\item
A path $(\sigma_0,\sigma_1,\ldots,\sigma_\ell)$ in $\Sigma(M)$ 
is said to be \emph{increasing} if $\sigma_{i-1} \cap \sigma_i 
\in \Facet^- \sigma_{i-1} \cap \Facet^+ \sigma_i$
for each $i \in \{1,\ldots,\ell\}$.
\item
A \emph{path} in $\rmN(M)$ is a sequence
$(v_0,v_1,\ldots,v_\ell)$ such that
$\{v_i\}$ is a vertex of $\rmN(M)$ for each $i \in \{0,\ldots,\ell\}$ and
$[v_{i-1},v_i]$ is an edge of $\rmN(M)$ for each $i \in \{1,\ldots,\ell\}$.
\item
A path $(v_0,v_1,\ldots,v_\ell)$ in $\rmN(M)$ 
is said to be \emph{increasing} if
$v_{i-1}<v_i$ for each $i \in \{1,\ldots,\ell\}$.
\end{enumerate}
\end{definition}

Then Theorem \ref{Sigma(M)=Sigma(N(M))}(b) and Lemma \ref{facet side} 
immediately give the following.

\begin{proposition}
Let $M \in \AA$.
Then there exist bijections
\begin{align*}
\{ \text{paths in $\Sigma(M)$} \} 
&\leftrightarrow \{ \text{paths in $\rmN(M)$} \}
\end{align*}
induced by Theorem \ref{Sigma(M)=Sigma(N(M))}(b).
Under this bijection, a path in $\Sigma(M)$ is increasing
if and only if the corresponding path in $\rmN(M)$ is increasing.
\end{proposition}

\end{document}